\documentclass[twoside,a4paper,11pt]{article}
\usepackage[top=4cm, bottom=4cm, left=3cm, right=3cm]{geometry}

\usepackage{macros}

\renewcommand{\phi}{\varphi}

\newcommand{\pr}{\mathrm{pr}}
\newcommand{\reg}{\mathrm{reg}}
\newcommand{\sing}{\mathrm{sing}}

\usepackage{amsmath}               
  {
      \theoremstyle{plain}
      \newtheorem{assumption}{Assumption}
      
  }

\begin{document}
\title{On the implicit regularization of Langevin dynamics with projected noise}
\author{
Govind Menon\thanks{Division of Applied Mathematics, Brown University and Institute for Advanced Study, Princeton
\texttt{govind\_menon@brown.edu}.}
\and
Austin J.\ Stromme\thanks{Department of Statistics, CREST, ENSAE, IP Paris, \texttt{austin.stromme@ensae.fr}.}
\and
Adrien Vacher\thanks{Department of Statistics, CREST, ENSAE, IP Paris, \texttt{adrien.vacher@ensae.fr}.}
}

\maketitle

\begin{abstract}
  We study Langevin dynamics with noise projected
    onto the directions orthogonal to an isometric group action. 
    This mathematical model is introduced to shed new light on 
    the effects of symmetry on stochastic gradient descent for over-parametrized models.
    Our main result
    identifies a novel form of implicit regularization:
    when the initial and target density are both invariant
    under the group action,
    Langevin dynamics with projected noise
    is equivalent in law
    to Langevin dynamics with isotropic diffusion but with an additional
    drift term proportional to the negative log volume
    of the group orbit.
    We prove this result
    by constructing a coupling
    of the two processes via
    a third process on the group itself,
    and identify
    the additional drift as the mean curvature of the orbits.
\end{abstract}

\section{Introduction}

A central feature of
modern machine learning is that models are often
heavily over-parameterized yet still
achieve excellent generalization performance.
For example, when the model has more parameters
than training data points,
classical statistical 
theory suggests that in this case one must add regularization to the training method to avoid overfitting.
However, this theory is widely known to conflict with practice, where
such models can and do generalize well without explicit
regularization~\cite{zhang2021understanding}.
Even when the model has fewer parameters than data points,
there is typically still over-parameterization coming from
the model architecture itself, in the sense
that there are parameter changes which do not affect the model output.

The by now widely accepted explanation for the fact
that over-parameterized models can generalize well
is that
the optimization methods used to train the model
{\it implicitly regularize} it by biasing it towards simpler solutions which generalize better~\cite{neyshabur2014search}.
However, despite significant
efforts,
the precise nature and mechanism of this implicit
regularization is not fully understood. 
And importantly, most theory for implicit regularization
applies either only to gradient descent or equally to stochastic
gradient descent (SGD), largely leaving open the question 
of the effect of stochasticity in particular on implicit regularization.
This is
a fundamental question due to the ubiquity of stochastic gradient
descent in practice.

In this work, we introduce a continuous
time model for SGD
in the case where the over-parameterization
arises from
the model architecture itself.
The primary motivation for our work
is the simple observation that
when the model is over-parameterized in this way,
{\it the stochastic gradients will only point in directions which are orthogonal to the over-parameterization}.

\paragraph*{Over-parameterization via group symmetries.}
Our main structural assumption is that
the over-parameterization can be described
by a group symmetry.
For concreteness, suppose that
$f(x) = \E_{z \sim \mc P}[L(x, z)]$ where $L(x ,z)$ is the loss
incurred on the data $z \sim \mc P$ with model parameter $x$.
We model the over-parameterization
by assuming that there is a group $G$ acting on $\R^d$
such that $L(x, z) = L(g\cdot x, z)$ for all $g \in G$
and $z \in \supp(P)$.
Define the orbit $\mc O_x := \{g\cdot x \colon g \in G\}$.
Then the training gradients $\{\nabla_x L(x, z) \colon z \in \supp(P)\}$
will all be orthogonal to the tangent space $T_x \mc O_x$,
and thus so will be any stochastic gradient.
Many deep learning models indeed have group symmetries.
For example, ReLU units are homogeneous,
and attention layers are invariant under
certain matrix multiplications.
Our setting includes, in particular,
some well-known simplified models
of over-parameterization reviewed in
Section~\ref{subsec:main_result}.

Analyzing implicit bias for SGD in this setting is challenging
because it requires a fine-grained
understanding of the training dynamics for highly non-convex
objectives.
We thus propose a simplified continuous-time
model of SGD which nonetheless permits us
to gain insight into the effect of noise
in directions
orthogonal to the over-parameterization.
In particular, let $P_x$ denote the orthogonal
projection on $(T_x \mc O_x)^{\perp}$, and 
$Q_x := I - P_x$ its complement.
We consider
\begin{equation}\label{eqn:projected_Langevin_continuous}
dX_t = -\nabla f(X_t) dt + \sqrt{2}(\alpha(X_t)P_{X_t}
+ \beta(X_t)Q_{X_t})dB_t,
\end{equation}
where $\alpha, \beta$ control the
strength of the projection onto the $P$ and $Q$ directions.

Observe that when $\alpha = \beta = 1$,
the diffusion matrix is the identity,
and so~\eqref{eqn:projected_Langevin_continuous}
reduces to standard Langevin dynamics, which 
is a common simplified model of stochastic gradient 
descent~\cite{raginsky2017non}; in this case the stationary distribution
is proportional to $e^{-f}$.
By contrast, when the diffusion matrix
is not the identity, the stationary distribution
can generally only be described indirectly as the
solution of a PDE.
The goal of this paper is to understand the effect on implicit regularization of noise in directions
orthogonal to the over-parameterization
by studying~\eqref{eqn:projected_Langevin_continuous}
in the case where
$\alpha \neq \beta$.

\paragraph*{Main contribution.}
Assuming that the group $G$ acts by isometries,
our main result, Theorem~\ref{thm:main}, establishes
a novel equivalence
between the dynamics~\eqref{eqn:projected_Langevin_continuous}
and different dynamics with an isotropic diffusion matrix but with an additional
entropic drift term.
In particular, we show that under certain
regularity assumptions on $f,\alpha$
and $\beta$,
and assuming that the dynamics are initialized at a $G$-invariant
measure, the dynamics~\eqref{eqn:projected_Langevin_continuous}
are equivalent in marginal law to the dynamics
$$
dY_t = -(\nabla f(Y_t) + (\alpha(Y_t)^2 - \beta(Y_t)^2)\nabla \log \vol \mc O_{Y_t})dt + \sqrt{2} \alpha(Y_t) dB_t,
$$ where the orbit volume $\vol \mc O_y$ is computed with respect
 to the induced Riemannian geometry on the orbit $\mc O_y$.
 In other words, the equivalent dynamics
 have an isotropic diffusion matrix but an additional
 drift term proportional to the {\it negative} gradient
 of the log volume of $\mc O_x$.
Our result reveals that having more noise in the
$P$ versus $Q$ directions leads to
a novel type of implicit regularization of the SDE~\eqref{eqn:projected_Langevin_continuous} towards
points with a small orbit as measured by their embedded volume.
Furthermore, our proofs identify that this new phenomenon is essentially
geometric: 
 the drift term $\nabla \log \vol \mc O_{x}$
 is exactly the negative mean curvature of the orbit $\mc O_x$ at $x$.
 In summary, our results shed new light on 
 the relationship
 between the architecture of a model
 and implicit regularization, by using
 tools from geometry and stochastic calculus to identify a novel type of
 symmetry-specific regularization.

\paragraph*{Organization of the paper.}
This paper is organized as follows. In the remainder
of this section we discuss related work.
In Section~\ref{sec:prelim}
we cover relevant background material
from geometry, the theory of group actions, and stochastic calculus.
In Section~\ref{sec:main_results}, we state our main result,
Theorem~\ref{thm:main}, provide discussion, and consider
several examples.
In Section~\ref{sec:proof_overview},
we give an overview of our proof of Theorem~\ref{thm:main}.
The Appendix then collects proofs that were omitted
from the main text.

\paragraph*{Related work.}
For models trained with a logistic-type loss,
there is a large literature dedicated
to understanding implicit bias by showing
convergence in direction
to max-margin solutions;
first for logistic regression~\cite{soudry2018implicit} and then for deep linear networks~\cite{ji2018gradient}, homogeneous
networks~\cite{lyu2019gradient},
and most recently for so-called nearly-homogeneous networks~\cite{cai2025implicit}.
The square loss setting is not as well understood,
with the prevailing model being deep linear neural networks.
An influential conjecture~\cite{gunasekar2017implicit}
states that certain two layer linear neural networks implicitly 
minimize nuclear norm;
this conjecture was established
under some conditions~\cite{li2018algorithmic}.
The work~\cite{arora2019implicit}
suggested that, for general deep linear networks, implicit regularization is not
captured by any matrix norm.

Comparatively less is known about
the effects of stochasticity in particular on implicit bias.
A widely known concept is
that stochastic algorithms favor so-called flat minima~\cite{keskar2016large}, which tend to generalize
better.
Due to the difficulty of developing a detailed
understanding of stochastic
gradient descent dynamics, a
common approach, and the one we take
in this paper, is to invoke
the central limit theorem
and resort to an SDE approximation~\cite{cheng2020stochastic}.
An early work studied the implicit bias of
such an SDE for linear regression~\cite{ali2020implicit}.
In the case of diagonal linear
neural networks, the implicit bias can be
exactly characterized~\cite{pesme2021implicit,even2023s}.
For an SDE model of
two layer linear neural network training,
formulas for the evolution of the
singular values and determinant
were derived in~\cite{varre2024sgd};
in particular, it was shown
that the determinant decreases deterministically, in contrast
to the gradient flow dynamics, suggesting
implicit regularization.

As we discuss more in Section~\ref{sec:prelim},
the symmetries of $G$ lead to a
natural quotient space $\R^d/G$,
and, away from certain degenerate orbits,
the quotient map is a Riemannian submersion.
Our work is therefore closely related to the study
of Langevin dynamics on manifolds in the presence
of a Riemannian submersion. 
Early work studied the image of Brownian
motion (with identity diffusion)
under a Riemannian submersion and identified the presence of a mean curvature correction in the image space~\cite{pauwels1990riemannian,carne1990geometry}.
The Riemannian
submersion machinery has been
extensively applied to analyze
the geometry of certain quotients of Euclidean
space in the literature on shape spaces,
see~\cite{le1993riemannian} and subsequent references.

This submersion approach was 
later used to construct Dyson Brownian motion~\cite{huang2023motion}
(see also Example~\ref{ex:eigenvalues}),
and Langevin dynamics on the Bures-Wasserstein
space~\cite{yu2023riemannian},
as well as to study the effective dynamics of
deep linear networks~\cite{menon2024geometry}.
In the settings
of both~\cite{huang2023motion,menon2024geometry}, it was found that
for $\alpha, \beta$ constant,
an SDE with anisotropic diffusion matrix $\alpha P + \beta Q$
projects to an SDE in the quotient space
with an additional drift of the form $-\beta^2\nabla \log \vol \mc O$
and diffusion scaled by $\alpha$.

Our work is inspired by the works~\cite{huang2023motion,menon2024geometry}, but differs
from them in the following important ways.
First of all, our main observation
Theorem~\ref{thm:main} is novel.
Secondly, rather than passing to the quotient
space $\R^d/G$, we work on
$\R^d$ and thus use entirely
different proof techniques.
One benefit of staying on $\R^d$ 
is that
it allows us to avoid technical 
issues coming from the fact that the quotient
space is, in general, a stratified singular
space (see Section~\ref{subsec:stratification}). 
Theorem~\ref{thm:main} is also free
of stopping times, while those previous
works required stopping times to deal
with the SDE hitting the boundary of the quotient space.
Another important benefit of working
on $\R^d$ is conceptual:
our coupling based proof of Theorem~\ref{thm:main}
in Section~\ref{sec:proof_overview}
provides a new and complementary intuition to that
coming from the submersion
picture which may be of further use.
Lastly,
as in these previous works,
we rely on
Proposition~\ref{prop:log_volume_mean_curvature},
connecting 
the mean curvature to the log-volume,
from~\cite{carne1990geometry}.

The concurrent work~\cite{anonymous2025quotientspace}
takes a submersion perspective similar
to~\cite{huang2023motion,menon2024geometry}
and obtains results closely related to our own.
Their main result states that,
in the presence of a Lie group $H$ with
proper, free isometric action
on a Riemannian manifold $M$,
Langevin dynamics with
projected noise is equivalent, on the quotient
space $M / H$, to the projection
of Langevin dynamics with identity diffusion
but an additional drift corresponding
to the mean curvature.
As we review in Section~\ref{subsec:stratification},
there are no non-trivial subgroups of $O(d)$
with free action
on $\R^d$, so their theory does not apply to our
setting.
Moreover, they neither conclude
equivalence at the level of $M$,
nor do they relate the mean curvature to the 
orbit volume, nor do they interpret
the result as implicit regularization.

 \paragraph*{Notation.}
 We write the set of $d \times d$ 
 real orthogonal matrices $O(d)$,
 and those with determinant one
 as $SO(d)$. When discussing an SDE
 of the form $dZ_t = -A(Z_t)dt + \sqrt{2}B(Z_t)dB_t$
 with $Z_t, A(Z_t) \in \R^k$ and
 $B(Z_t) \in \R^{k \times k}$,
 we refer to $B(Z_t)$ as the {\it diffusion matrix}.
 For a random variable
 $X \in \R^d$,
 we write its law $\mathcal{L}(X)$.
\section{Preliminaries}

\label{sec:prelim}

In this section, we collect some preliminary material
on isometric group actions on $\R^d$, the orbit volume,
and stochastic calculus.
For general background in Riemannian geometry and the basics
of Lie group theory,
we refer to the books~\cite{lee2013introduction,lee2018introduction,hall2013lie}.
For background on stochastic calculus, we refer the reader
to the books~\cite{hsu2002stochastic,oksendal2013stochastic}.

\subsection{Isometric group action assumption
and background on Lie groups}
\label{subsec:group_actions}

For a thorough introduction to Lie group theory,
we refer to the textbook~\cite{hall2013lie}.
A Lie group $G$ is a smooth manifold with
a group structure such that the group operations are smooth.
The only non-trivial fact we shall use from Lie group theory
is that a compact Lie group $G$ has a Haar measure:
a measure $\mu$, invariant under right and left-multiplication
by group elements, which is unique up to scaling~\cite[Proposition 16.10]{lee2013introduction}.
We write the unique Haar measure with
$\int_G \mu = 1$ as $\mathrm{unif}_G$.

Throughout this work, we make the following assumption.
\begin{assumption}[$f$ is invariant under an isometric group action]
\label{assum:isometric}
We assume that there exists a closed Lie group $G\subset O(d)$ acting
by isometries on $\R^d$, such that $f(g\cdot x) = f(x)$ for
all $g \in G$.
\end{assumption}
For simplicity, we will identify $G$ with its matrix
representation, so one may concretely think of $G$
as an embedded submanifold of matrices $G \subset \R^{d \times d}$,
closed under matrix multiplication and inversion,
such that $gg^\top = I$ for all $g \in G$.
The projection matrix $P$ is defined as
the orthogonal projection onto $(T_x \mc O_x)^{\perp}$.

\subsection{Stratification of $\R^d$ by orbit type}
\label{subsec:stratification}
In this work we will use the theory
of {\it orbit types} to deal with
certain singular orbits, and to this
end here recall some basic
elements from the theory of group actions
on Riemannian manifolds, following~\cite[Chapter 2]{berndt2016submanifolds}.

The {\it orbit} of a point $x \in \R^d$ is defined
by
$$
\mc O_x := \{g\cdot x \colon g \in G\}.
$$
The stabilizer of a point $x \in \R^d$
is
$$
G_x :=\{ g \in G \colon g\cdot x = x\}.
$$
For compact groups acting
on manifolds, the best case
is when the group acts freely, meaning that 
the stabilizer is the identity, $G_x = \{e\}$,
for all $x$~\cite[Theorem 21.10]{lee2018introduction}.
The main benefit of free actions
in this case
is that the quotient space can be given a Riemannian
structure such that the projection map is a Riemannian
submersion.

Unfortunately, since the stabilizer of $0$ is
the whole group, there
are no non-trivial subgroups of $O(d)$ with free action on $\R^d$.
The lack of free actions means
that the quotient space $\R^d/G$ will always have a boundary,
and potentially a singular boundary,
which rules out a simple Riemannian submersion picture (see also Example~\ref{ex:bures--wasserstein}).
In practical terms, this means that the
matrices $P, Q$ and the volume $\vol \mc O_x$
will not be smooth on all of $\R^d$, having instead
singularities, which then must be dealt with to establish
rigorous theory.

These singularities
can be succinctly described
by considering {\it orbit types}.
For two orbits $\mc O_x$ and $\mc O_y$,
we say they have the same type if and only if 
$G_x$ and $G_y$ are conjugate.
This yields a partial order
by defining $\mc O_x < \mc O_y$
if and only if $G_x$ is conjugate to a subgroup
of $G_y$.
With this partial order, there is a unique
minimal orbit type, the {\it principal}
orbit type. These orbits have maximal dimension
and form an open and dense subset of $\R^d$~\cite[Proposition 2.2.4]{berndt2016submanifolds}.
Orbit types which have maximal dimension but are not
principal are termed {\it exceptional},
and the remaining orbit types are {\it singular}.

We write the set of principal orbits as $\R^d_{\pr}$,
and the set of singular orbits as $\R^d_{\sing}$.
We call the union of the principal and exceptional orbits
the ``regular" orbits, and write it as $\R^d_{\reg}$.
Notice that $\R^d_{\reg}$
is exactly the set of orbits of maximal dimension.

The main fact we shall require about orbit types
is that the projection operator $P$ 
(and thus $Q$) is smooth on the regular orbits $\R^d_{\reg}$.
\begin{prop}[Smoothness of $P$ and $Q$ away from
singular orbits]
\label{prop:smoothness_of_P}
Suppose $x \in \R^d_{\reg}$.
Then there is a smooth orthonormal frame $V_1, \ldots, V_m$,
defined on a neighborhood of $x$, such that for 
each $y \in U$, $V_1, \ldots, V_m$ spans $T_y\mc O_y$.
In particular, $P$ and $Q$ are smooth in a neighborhood of
$x$, and $\R^d_{\reg}$ is an open set.
\end{prop}
The proof of this result is included in Appendix~\ref{appendix:prelim_proofs}.

\subsection{Mean curvature and orbit volume}
For a regular orbit $\mc O_x$, we write its volume as an 
embedded submanifold of $\R^d$ as 
$\vol \mc O_x$, and then define $\vol \mc O_x := 0$ 
if $\mc O_x$ is a singular orbit.
The orbit volume $\vol \mc O_x$ is closely related to
the geometry of the orbit as an embedded Riemannian
manifold, and in particular to its mean curvature.
Recall that for an embedded Riemannian manifold $N \subset \R^d$,
its second fundamental form
is a tensor which maps $v, w \in T_xN$ to 
$$
h^N_x(v, w) := P^N_x(\nabla_{V} W ),
$$ where $V, W$ are smooth extensions
of $v, w$ respectively, and $P^N_x$ is the projection
onto the normal space $(T_x N)^{\perp}$ at $x$.
Using the second fundamental form, one can define
the {\it mean curvature}, a fundamental
extrinsic geometric object.
Given an orthonormal basis $v_1, \ldots, v_m$ 
for $T_xN$,
the mean curvature is defined to be
$$
H^N(x) := \sum_{i =1 }^m h_x^N(v_i, v_i).
$$ When $N$ is an orbit $\mc O_x$, we will generally
suppress it and simply write the mean curvature
and second fundamental form as $H(x)$ and $h_x$, respectively.

We then have the following important observation on the regularity
 of $\vol \mc O_x$ and its connection to mean curvature, which seems
 to have first appeared in~\cite{carne1990geometry}, see also~\cite{pacini2003mean}.
\begin{prop}[Smoothness and gradient of the log-volume]\label{prop:log_volume_mean_curvature}
The orbit volume
$\vol \mc O_x$ is smooth on the regular orbits $\R^d_{\reg}$.
Moreover, on the regular orbits $\R^d_{\reg}$, we have the identity
\begin{equation}
\label{eqn:mc-entropy}
H(x)= 
-\nabla \log \vol \mathcal{O}_x,
\end{equation}
where $H(x)$ is the mean curvature vector of $\mc O_x$ at $x$.
\end{prop}

We give a complete proof
of this result
in Appendix~\ref{appendix:prelim_proofs}.

\section{Main result}
\label{sec:main_results}

\subsection{Main result and examples}
\label{subsec:main_result}
Recall that we assume $f \colon \R^d
 \to \R$ is a smooth
 function and $G$ is a Lie group
 such that $f$ and $G$
 satisfy Assumption~\ref{assum:isometric}.
 We write the orthogonal projection
 onto $(T_x \mc O_x)^{\perp}$ as $P_x$ and its complement
 as $Q_x := I - P_x$.
 We consider the SDE
\begin{equation}\label{eqn:projected_SDE_full}
dX_t = -\nabla f(X_t)dt + \sqrt{2}(\alpha(X_t)P_{X_t}
+ \beta(X_t)Q_{X_t})dB_t,
\end{equation}
where $\alpha \colon \R^d \to \R_{\geqslant 0}$ and $\beta \colon \R^d \to \R_{\geqslant 0}$
are non-negative functions which
control the strength of
the noise in the horizontal
versus vertical directions.
Our main result applies to the solution of the SDE
\begin{equation}\label{eqn:full_isotropic_Langevin}
dY_t = -(\nabla f(Y_t) + (\alpha(Y_t)^2
-\beta(Y_t)^2)\nabla \log \vol \mc O_{Y_t})dt + \sqrt{2}\alpha(Y_t)dB_t.
\end{equation}
\begin{theorem}[Main result]\label{thm:main}
Suppose Assumption~\ref{assum:isometric} holds.
Suppose $f, \alpha, \beta$
are smooth and $G$-invariant,
$\nabla f$ and
$\alpha$ are
globally Lipschitz, $\alpha$ is positive everywhere,
and the difference $\alpha - \beta$ 
is compactly supported
within the set of regular orbits $\R^d_{\reg}$.
Suppose further
that the SDEs~\eqref{eqn:projected_SDE_full}
and~\eqref{eqn:full_isotropic_Langevin}
are initialized at a $G$-invariant
distribution with finite second
moment.
Then both~\eqref{eqn:projected_SDE_full}
and~\eqref{eqn:full_isotropic_Langevin}
have unique strong solutions
$(X_t)_t, (Y_t)_t$,
and $\mathcal{L}(X_t) = \mathcal{L}(Y_t)$ 
for all $t \geqslant 0$.
\end{theorem}

 We emphasize several aspects of the above result. First and most importantly,
the equivalent SDE has an additional drift term proportional
to the {\it negative} gradient of the log volume of $\mc O_x$.
This volume is measured with respect to the geometry 
of $\mc O_x$ as an embedded manifold of $\R^d$,
and is a novel type of implicit regularization for the SDE~\eqref{eqn:projected_Langevin_continuous}, biasing
the SDE towards points with a small orbit. As noted above, the $\nabla \log \vol \mc O_x$
term is fundamentally geometric: it is exactly the (negative) mean curvature
of the orbits $\mc O_x$. Second, although the SDE~\eqref{eqn:projected_SDE_full}
has non-isotropic noise, the equivalent SDE~\eqref{eqn:full_isotropic_Langevin} has isotropic
diffusion.
Third, we assume that the initial distribution is invariant under $G$.
This is necessary for our results, but is mild, as it is satisfied
by initializing
at a normal distribution with isotropic covariance. 
Fourth, we assume that $\alpha - \beta$ is compactly
supported within the set of regular orbits $\R^d_{\reg}$.
This assumption is necessary to apply standard strong
well-posedness results for SDEs, as it allows us
to avoid the non-smoothness of $P, Q$ and $\nabla \log \vol \mc O_x$ at the singular
orbits.

In the case where $\alpha$ is larger than $\beta$,
Theorem~\ref{thm:main} reveals that the SDE~\eqref{eqn:projected_SDE_full}
is biased towards points with a small orbit.
This identifies an intriguing new type
of implicit regularization based on the volume
$\vol \mc O_x$ rather than a more standard norm
(see the examples below).
The implicit regularization in our
model~\eqref{eqn:projected_SDE_full}
is thus tightly
connected to the symmetries of the over-parameterization,
and therefore the model architecture,
in the sense that different architectures
will lead to different symmetry groups with different
$\nabla \log \vol \mc O_x$ terms.
Theorem~\ref{thm:main} therefore suggests the intriguing possibility that when
one is selecting a particular model architecture one is also,
implicitly,
selecting a particular regularizer.

The intuition behind Theorem~\ref{thm:main}
is based on the observation
that, because~\eqref{eqn:projected_SDE_full}
is initialized at a $G$-invariant
distribution,
it will remain $G$-invariant
at all times.
Hence for any process
$g_t \in G$ such that 
$g_t \cdot X_t$ remains
$G$-invariant, the
marginal law of $g_t \cdot X_t$
must be that of $X_t$.
By choosing $g_t$ appropriately, we can thus
introduce additional noise in the $Q$ directions
without changing the marginal law.
The $\nabla\log \vol \mc O_x$ term
is the mean curvature of $\mc O_x$
at $x$ (Proposition~\ref{prop:log_volume_mean_curvature}),
and arises from the constraint
that $g_t$ remain in $G$.
This proof strategy
is developed in detail in Section~\ref{sec:proof_overview}.

Let us now consider several illustrative examples.
\begin{example}[Radial symmetries]
The simplest setting is when $G = SO(d)$
acts on $\R^d$ via $U \cdot x = Ux$.
The orbits are the spheres $\mc O_x = \S^{d -1}(\|x\|)$,
 $\R^d_{\reg} = \R^d\setminus \{0\}$,
and the volume is $\vol \mc O_x = c_d\|x\|^{d - 1}$
for a dimension-dependent constant $c_d$.
Functions $f$ which are symmetric under this
action are radial.
\end{example}
\begin{example}[Projection onto eigenvalues]
\label{ex:eigenvalues}
Consider the conjugation by $O(d)$ over symmetric matrices identified with $\mathbb{R}^{d(d+1)/2}$: for all $O \in O(d)$ and $M \in \text{Sym}(d)$, $O \cdot M = O^\top M O$. Two matrices are in the same orbit if and only if they have the same eigenvalues.
The set of regular orbits $\R^d_{\reg}$
is then exactly the matrices with distinct
eigenvalues.
 In Appendix~\ref{sec:examples},
 we show that for $M$ with eigenvalues $(\lambda_1, \cdots, \lambda_d)$, 
$
\vol \mc O_M = c_d \prod_{i<j} |\lambda_i - \lambda_j| \, ,
$
for a dimension-dependent constant $c_d$.
The class of functions which are symmetric under this
action are the spectral functions.
\end{example}

\begin{example}[Bures--Wasserstein case]
\label{ex:bures--wasserstein}
Consider the right multiplication by $O(d)$ over 
$d\times d$ real matrices identified with $\mathbb{R}^{d^2}$: for all $O \in O(d)$ and $M \in \mathbb{R}^{d\times d}$, $O \cdot M = M O$. Two matrices $M_1, M_2$ are in the same orbit if and only if $M_1^\top M_1 = M_2^\top M_2$.
It is straightforward to verify that the
regular orbits are rank $d$ matrices,
the exceptional orbits are rank $d - 1$
matrices, and singular orbits
are matrices with rank strictly less than $d -1$.
 Furthermore, we show in Appendix~\ref{sec:examples} that, denoting $(\sigma_1, \ldots, \sigma_d)$ the singular values of $M$, the volume of the orbit at $M$ is given by $\vol \mc O_M = c_d \prod_{i<j} \sqrt{\sigma_i^2 + \sigma_j^2}$ with $c_d$ a dimension-dependent constant; related formulas appeared in~\cite{carne1990geometry,huang-cp,yu2023riemannian}.
 In particular, the orbit volume is singular precisely
 at the singular orbits.

 This group action is closely related to the
 Bures--Wasserstein manifold on
 positive-definite matrices, as well as the 
 study of over-parameterization in linear
 neural networks. Indeed,
 the quotient space
$\R^{d^2}/O(d)$ is exactly the Bures--Wasserstein
manifold, which is known
to have a singular boundary~\cite{massart2020quotient}.
Functions $f$ which are symmetric under this
action are those of the form $f(X) = g(XX^\top)$.
Such functions
have been studied as a model of over-parameterization
in neural networks since the influential
work~\cite{gunasekar2017implicit},
see for example~\cite{li2018algorithmic,arora2019implicit}.

Finally, we remark here that the
parameterization $XX^\top$ can be generalized to
the deep linear network~\cite{arora2018optimization},
where one instead considers
a product of $N$ real matrices 
$X_NX_{N-1}\cdots X_1$.
This model has $\GL_d(\R)$
symmetries, but one may nonetheless
apply our model to this case by considering
only the subgroup of $O(d)$ symmetries, see also~\cite{menon2025entropy}.
\end{example}

\subsection{Identity diffusion and inverse log volume
factors}
\label{subsec:applications}

To better illustrate Theorem~\ref{thm:main},
it is useful to consider the case where $\alpha \equiv 1$
and
$\beta$ is a function of $\log \vol\mc O_y$, since in this
case the stationary
distribution of the SDE is explicit.
\begin{corollary}[Identity diffusion]\label{cor:identity_diffusion}
Suppose Asssumption~\ref{assum:isometric} holds,
$f$ is smooth and $\nabla f$ is globally Lipschitz.
Consider the SDE
\begin{equation}\label{eqn:projected_SDE_identity}
dX_t = -\nabla f(X_t)dt + \sqrt{2}(P_{X_t} + \beta(X_t)Q_{X_t})dB_t,
\end{equation} where $\beta(y) = \phi(\log \vol \mc O_y)$
for a smooth function
$\phi \colon \R \to \R_{\geqslant 0}$.
Suppose that $\phi(\R) \subset [c, C]$
for $0 < c < C$ and
$\phi$ is $1$ on 
a ray $(-\infty, \tau_0]$.
Suppose further that the~\eqref{eqn:projected_SDE_identity}
is initialized at a $G$-invariant distribution
with finite second moment.
Then there is a unique strong solution to~\eqref{eqn:projected_SDE_identity},
and it has stationary distribution
$$
\rho \propto e^{-\int_{0}^{\log \vol \mc O_y} (1- \phi(s)^2) d s} e^{-f(y)}.
$$
\end{corollary}
In particular, if $\phi =\eps > 0$
on a ray $[\tau_1, \infty)$ with $\tau_1 < 0$, then for any $y$ such that
$\log \vol \mc O_y \geqslant \tau_1$,
$\rho(y) = \frac{1}{A}(\vol \mc O_y)^{\eps^2-1}e^{-f(y)}$,
where $A$ is a normalizing constant. The regularizing effect of the log volume
therefore manifests as an inverse factor
in the stationary distribution.

This Corollary imposes slightly different
regularity assumptions on $\beta$ than Theorem~\ref{thm:main},
so we explain the minor modifications
which must be made to
the proof of Theorem~\ref{thm:main}
in Appendix~\ref{appendix:proof_by_PDE}.

Let us finally remark that, ignoring regularity issues, Corollary~\ref{cor:identity_diffusion}
suggests that the fully projected SDE
\begin{equation}\label{eqn:ex_fully_projected}
dX_t = -\nabla f(X_t) dt + \sqrt{2}P_{X_t}dB_t,
\end{equation} has stationary distribution proportional
to $\frac{1}{\vol \mc O_x}e^{-f}$.
However,
to the best of our knowledge,
the well-posedness
of the dynamics~\eqref{eqn:ex_fully_projected}
falls outside known
results in SDE theory due to the
non-smoothness of the projection
$P$ at singular orbits, or equivalently
the singularities in the drift
of the equivalent dynamics.
The only result we are aware of 
along these lines uses an analysis
tailored to shape spaces in certain parameter regimes~\cite{le1994brownian}.
We therefore leave the rigorous study of~\eqref{eqn:ex_fully_projected}
for general groups $G$
as an interesting direction for future work.

\section{Proof overview: coupling the two SDEs}
\label{sec:proof_overview}

In this section we overview
our proof approach, which is based on
introducing a third stochastic
process $g_t$ on the group $G$ 
such that $g_t \cdot X_t$ is a solution to~\eqref{eqn:full_isotropic_Langevin}.
Full details for this proof are provided
in Appendix~\ref{appendix:SDE_proof}.

\subsection{$G$-invariance of the SDE and proof approach}
The first key idea is that, so long as the SDEs
are initialized at a $G$-invariant distribution, they
must remain
$G$-invariant for all times. This observation
is formalized in the following Lemma,
proved in Appendix~\ref{appendix:dynamics_remain_g_invariant}.
\begin{lemma}[SDE remains $G$-invariant]
\label{lem:remain_G_invariant}
Suppose $f, \alpha, \beta$
are as in Theorem~\ref{thm:main}
and Assumption~\ref{assum:isometric} holds.
If~\eqref{eqn:projected_SDE_full}
is initialized at a distribution which is $G$-invariant,
then the law of $X_t$ remains $G$-invariant
at all times.
Similarly, if~\eqref{eqn:full_isotropic_Langevin}
is initialized at a distribution which is $G$-invariant,
then the law of $Y_t$ remains $G$ invariant at all times.
\end{lemma}

The idea of the proof of Theorem~\ref{thm:main} is to use the $G$-invariance
of Lemma~\ref{lem:remain_G_invariant}
to design a process $g_t$ with the following properties:
\begin{enumerate}
    \item[P1.] the process $(g_t)_t$ remains on the Lie group $G$, and
    \item[P2.] the process $(g_t \cdot X_t)_t$ is a weak solution to
    the SDE~\eqref{eqn:full_isotropic_Langevin}.
\end{enumerate}
Given $g_t$ with the above properties, Lemma~\ref{lem:remain_G_invariant}
then implies that $(g_t \cdot X_t)_t$ has the same marginal
law as $X_t$,
and thus yields Theorem~\ref{thm:main}.
We essentially introduce
rotation to the particle $X_t$ via
a process $g_t$, which will not affect
the marginal law, but will still introduce
 noise in the $Q$ directions
 so as to make $g_t \cdot X_t$ match
 with $Y_t$ from~\eqref{eqn:full_isotropic_Langevin},
 see Figure~\ref{fig:proof_idea}.

 \begin{figure}
     \centering
     \includegraphics[width=.8\linewidth, trim=0cm 4cm 0cm 0cm, clip]{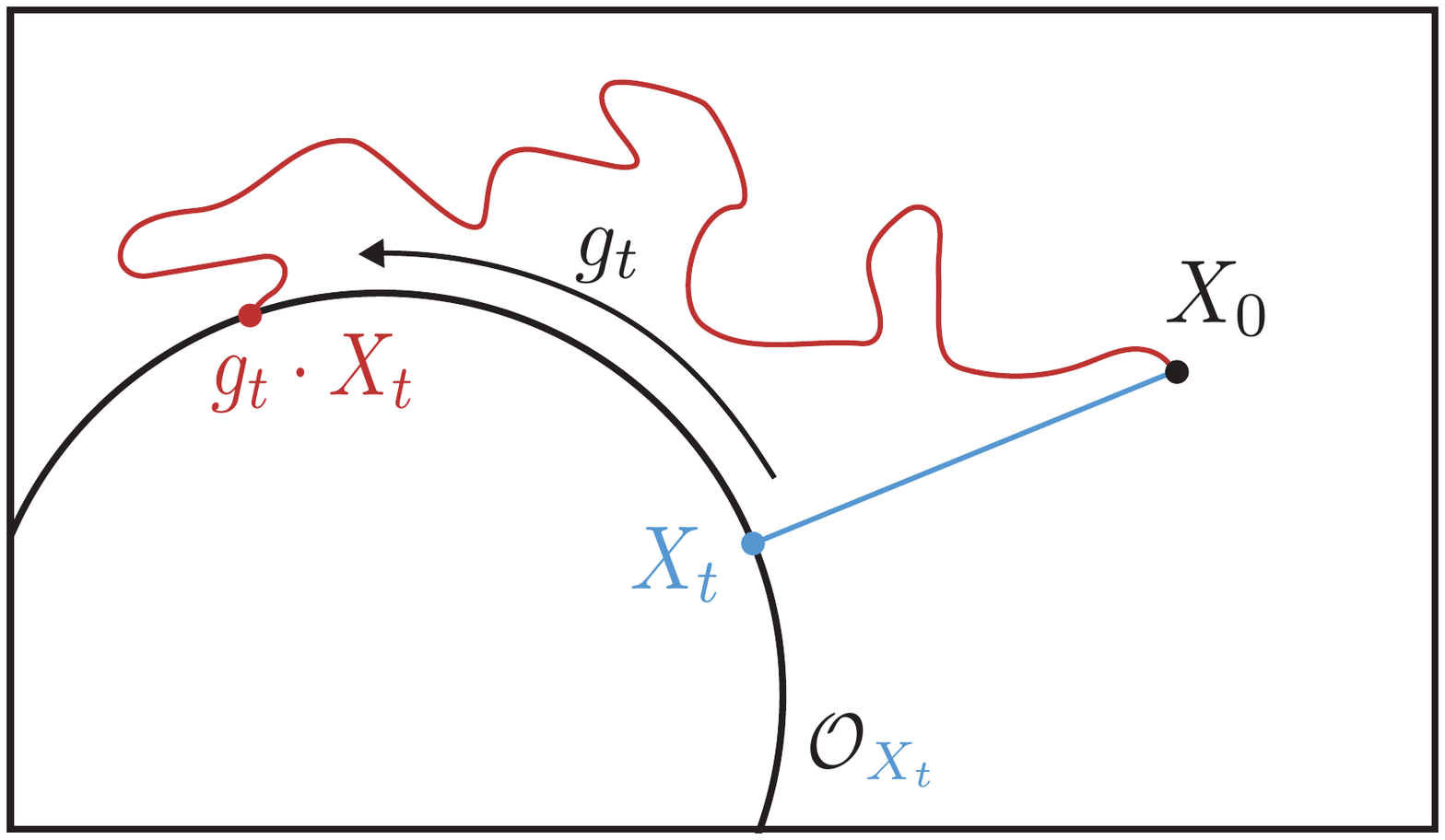}
     \caption{By introducing an appropriate process $g_t \in G$, we can create additional
     movement in the $T_x \mc O_x$ directions without changing the $G$-invariance property, and thus the marginal distributions.}
     \label{fig:proof_idea}
 \end{figure}

\subsection{Constructing Brownian motion on an orbit} To better explain our 
construction, we consider a simpler but related question:
given a {\it fixed} point $x$, how can we construct
uniform Brownian motion on the group orbit $\mc O_x$
of $x$ by a process of the form $g_t \cdot x$ for $g_t$ a stochastic
process in $G$?
 We thus solve a particular ``stochastic geometric control theory" problem.
Interestingly, it turns out that the answer
to this question is {\it not} to take $g_t$ as
uniform Brownian motion on the group $G$ itself (see Appendix~\ref{subsec:BM_on_group_counterexample} for a simple example).

To approach this question, it is useful
to recall the Itô form
of Brownian motion on a manifold embedded in
$\R^d$~\cite{lewis1986brownian,stroock2000introduction,inauen2023}.
\begin{theorem}[Brownian motion on an embedded manifold]
\label{thm:embedded_manifold_BM}
Suppose $M \subset \R^{l}$ is complete and connected
smooth manifold without boundary embedded in $\R^l$. For all $p \in M$, let $H^M(p) \in N_pM$
be the mean curvature vector at $p$,
and let $Q^M_{p}$ be the projection from $T_p\R^d$
to $T_pM$.
Then the SDE
\begin{equation}\label{eqn:embedded_manifold_BM}
dZ_t = H^M(Z_t)dt + \sqrt{2}Q^M_{Z_t}dB_t, \quad \quad Z_0 = p_0
\end{equation} has a unique strong solution
and is a Brownian motion on $M$ initialized at $p_0 \in M$.
\end{theorem}

The idea is then to try and find a process $g_t$ in $G$
such that $g_t \cdot x$ solves~\eqref{eqn:embedded_manifold_BM}
for $M = \mc O_x$.
Recall that we regard $g_t \in G$ as an
element of $\R^{d\times d}$. Suppose $g_t$ is following
a diffusion with (to be specified) drift $V$ and diffusion matrix $J$:
\begin{equation}\label{eqn:candidate_g_diffusion}
dg_t = V(g_t)dt + \sqrt{2} J(g_t) dB_t' \quad \quad g_0 = I,
\end{equation} where $B_t'$ is an independent Brownian motion on the space $\R^{d \times d}$.

For ease of notation, let $F(g, x) := g \cdot x$,
and let $L_{g, x} \colon T_g\R^{d \times d}
\to T_{g \cdot x} \mc O_x$ be
the linear map
\begin{equation}\label{eqn:Lgx_defn}
L_{g, x}A := dF_{g, x}[Q_{g}^GA, 0],
\end{equation}
where $Q_g^G$ is the projection
onto the tangent space of $G$ at $g$.
Assuming that $\mathrm{im}(J(g)) \subset T_gG$,
Itô's lemma then implies that
 $Y_t := F(g_t, x) = g_t \cdot x$ follows
\begin{equation}\label{eqn:image_g_diffusion}
dY_t = L_{g_t, x}V(g_t)dt + \sqrt{2}L_{g_t, x}J(g_t)dB_t'.
\end{equation}
Using the fact that $H(x) = \nabla \log \vol \mc O_x$
from Proposition~\ref{prop:log_volume_mean_curvature},
combined with Theorem~\ref{thm:BM_on_orbit},
we know that
Brownian motion on $\mc O_x$ can be
constructed with the following SDE:
\begin{equation}\label{eqn:BM_on_orbit}
dZ_t = -\nabla \log \vol \mc O_x dt + \sqrt{2}Q_{Z_t} dB_t'',
\end{equation} where $H(Z_t)$ is the mean curvature
of the orbit $\mc O_x$.
We thus try and match terms between~\eqref{eqn:image_g_diffusion}
and~\eqref{eqn:BM_on_orbit}.

Beginning with the diffusion terms, this suggests we take
$J$ such that
$$
L_{g_t, x}J(g_t)J(g_t)^\top L_{g_t, x}^\top = Q_{g_t \cdot x}.
$$ 
This can be ensured by taking
$$
J_0(g_t) := (L_{g_t, x}^\top L_{g_t,x})^{\dagger/2},
$$ where we use $(\cdot)^{\dagger}$ to denote the 
pseudo-inverse. 

Although the range of $J_0$ is always
in $T_gG$, this doesn't yet guarantee
that the process will remain in $G$ due
to the Itô correction.
This is similar to the standard construction of Brownian
motion on embedded manifolds encapsulated
in Theorem~\ref{thm:embedded_manifold_BM}: even if the drift
and noise are tangential to the manifold, the Itô terms mean
that the particle could still ``fall off" the manifold,
and therefore
one must add certain 
terms involving the second fundamental form of $G$ to stay on the manifold.
Let $h^G_g$ denote the second fundamental form of $G$ (as an embedded
manifold of $\R^{d \times d}$) at $g \in G$. Then put
$$
V_0(g) := \tr(h^G_{g}( L_{g, x}^{\dagger}\cdot, L_{g, x}^{\dagger}\cdot)),
$$ where we take the trace over an orthonormal basis of $T_{g \cdot x}\mc O_x$; note that $V_0(g) \in \R^{d\times d}$
is a matrix not a scalar.
With this definition, it can be shown that~\eqref{eqn:candidate_g_diffusion},
with $V = V_0$ and $J = J_0$ as above, remains on $G$. 

The following observation, proved in Appendix~\ref{subsec:second_fund_forms_eq}, shows that
this drift term is closely related to the second fundamental
form, and thus the mean curvature, of the orbit $\mc O_x$.
\begin{lemma}[Relationship between second fundamental forms]
\label{lem:second_fund_forms_eq}
Let $g \in G$ and $x \in \R^d$.
Let $h_{g \cdot x}^{\mc O_x}$ be the second
fundamental form of $\mc O_x$ at $g \cdot x$
and $h^G_g$ be the second fundamental form
of $G$ at $g$.
Then for all $A \in T_{g} G$ we have
$$
h_{g \cdot x}^{\mc O_x} (L_{g, x}A,L_{g,x}A) = 
P_{g \cdot x}(h^G_g(A, A) x).
$$
\end{lemma}
This result implies that the mean curvature $H(g\cdot x)$
of $\mc O_x$ at $g \cdot x$ satisfies
$
H(g\cdot x) = P_{g \cdot x} (V_0(g)x).
$
Taking $V = V_0$ and $J = J_0$, and
applying Proposition~\ref{prop:log_volume_mean_curvature}
we obtain
\begin{align*}
dY_t = d(g_t \cdot x) = 
(-\nabla \log \vol \mc O_{g_t \cdot x} - Q_{g_t \cdot x}(V_0(g_t)x)))dt + \sqrt{2} L_{g_t, x}J(g_t)dB_t'.
\end{align*}
We therefore nearly have what we wanted: it only remains
to get rid of the $Q_{g_t \cdot x} (V_0(g_t)x)$ term.
But because this term is in $T_x \mc O_x$,
it is in the image of $ L_{g_t, x}$,
and so
we can remove it by introducing
$$
V_1(g_t) :=-L_{g_t, x}^{\dagger}Q_{g_t \cdot x}(V_0(g_t)x).
$$ Since $V_1$ is always in $T_{g_t}G$, including it in the drift will not
violate property 1 above. The following result
sums up this discussion. A detailed proof is provided in Appendix~\ref{appendix:SDE_proof}.
\begin{theorem}[Brownian motion on an orbit]
\label{thm:BM_on_orbit}
Let $J_0, V_0, V_1$ be as above, and consider~\eqref{eqn:candidate_g_diffusion} 
with $V= V_0 + V_1$ and $J = J_0$. Then $g_t$ 
is a well-defined stochastic process on $\mathbb{R}^{d \times d}$,
which remains in $G$
at all times, and
$g_t \cdot x$ is a
Brownian motion on $\mc O_x$ initialized at $x$.
\end{theorem}

To the best of our knowledge, the above construction
of Brownian motion on the orbit $\mc O_x$ is novel.
\subsection{The full construction}
The full construction uses the same ideas as the
above with one significant difference.
The above proof ansatz can only {\it add} noise in
$Q$ directions, while Theorem~\ref{thm:main}
states an equivalence between processes with {\it a priori} 
different amounts of noise in the $Q$ directions.
The full construction gets around this by proving
equivalence of the 
SDEs~\eqref{eqn:projected_SDE_full}
and~\eqref{eqn:full_isotropic_Langevin} 
to a third process, with more noise on $Q$ directions than
either~\eqref{eqn:projected_SDE_full}
and~\eqref{eqn:full_isotropic_Langevin}.
Finally, once the full process $g_t$ has been constructed, Theorem~\ref{thm:main} follows
 from Lemma~\ref{lem:remain_G_invariant},
 which (as in the PDE proof) allows us to reduce
 to testing against $G$-invariant test functions.
 Full details are provided in Appendix~\ref{appendix:SDE_proof}.

 \paragraph*{Alternative proof by PDE analysis.}
 In Appendix~\ref{appendix:proof_by_PDE}
 we also provide an alternative proof
 by a PDE based analysis.
 The main advantage of this proof is that
 it only
 requires
 well-posedness of the corresponding weak
 Fokker--Planck equations and is rather direct.
 Nonetheless, it does not provide
 a conceptual explanation for the equivalence
 of Theorem~\ref{thm:main} as does the above 
 approach. It is also interesting
 to note that the core geometric
 fact underlying the proof by PDE
 analysis, Lemma~\ref{lem:hessian_G_invariant},
 is apparently completely distinct from that
 of the proof by SDE coupling, Lemma~\ref{lem:second_fund_forms_eq}.

\section{Discussion}

Motivated by the study of SGD for over-parameterized
models, in this paper we analyzed the SDE~\eqref{eqn:projected_SDE_full},
which features a diffusion matrix with
noise projected by different
amounts in normal versus tangential directions
to an isometric group action. 
Our main result, Theorem~\ref{thm:main},
identifies a novel type of implicit regularization
for this SDE, where particles are pushed
towards orbits with small volume.
We proved this result by constructing
a coupling of the SDE~\eqref{eqn:projected_SDE_full}
with a second
SDE on the group $G$,
and critically using the fact that the mean curvature of the orbits
is the negative gradient of the log-volume,
Proposition~\ref{prop:log_volume_mean_curvature}.
Overall, the results in this work
provide evidence for a close connection between
the symmetries of a
model architecture and its
implicit regularization,
and we believe
that further investigations into this connection
is an exciting direction for future research.
\paragraph*{Acknowledgements.}
GM was supported by the NSF grant DMS 2407055 and the Erik Ellentuck Fellow Fund at the Institute for Advanced Study, Princeton. 
\bibliographystyle{alpha}
\bibliography{annot}
\appendix 
\section{Omitted proofs from Section~\ref{sec:prelim}}
\label{appendix:prelim_proofs}

\begin{proof}[Proof of Proposition~\ref{prop:smoothness_of_P}] Let $A_1, \ldots, A_m \in \mathfrak{g}$
be such that $A_1 x, \ldots, A_m x$
span $T_x\mc O_x$. For $y \in \R^d$, consider
the vector fields $V_1(y) := A_1  y,\ldots,  V_m(y) :=
A_m y.$ 
We first claim that by continuity,
$(V_i(y))_{i = 1}^m$ must 
remain linearly independent for $y$ sufficiently
close to $x$. Indeed, consider the $d \times m$ 
matrix $B_{y}$ whose $i$th column is $V_i(y)$.
Then $B_{y}$ has rank $m$ if and only
if $\det(B_{y}^\top B_{y}) \neq 0$, and the latter is a continuous
condition which holds at $y = x$, so must hold in a neighborhood
around $x$.

Therefore, $(V_i(y))_{i = 1}^m$ must 
remain linearly independent for $y$ in some neighborhood
of $x$. Notice that $V_i(y) \in T_y\mc O_y$
for each $i$, and by the fact that $m$ is the maximal orbit
dimension they must in fact span $T_y \mc O_y$.
We may then take their Graham-Schmidt orthonormalization to obtain
$(W_i(y))_{i= 1}^m$ which are orthonormal and 
span $T_y \mc O_y$. Because the $V_i$ are smooth functions of $y$, so are $W_i$,
implying the first claim.
But we can write
$$
Q_y = \sum_{i =1}^m W_i(y)W_i(y)^\top, \quad \quad 
P_y = I - Q_y,
$$ yielding the result.
\end{proof}

\begin{proof}[Proof of Proposition~\ref{prop:log_volume_mean_curvature}]
We first establish smoothness,
before turning to the identity $H_x = \nabla\log \vol \mc O_x$.

To show that $\vol \mc O_x$ is smooth,
we first claim that there is a unique smooth structure and
geometry on $\pi(\R^d_{\pr}) \subset \R^d/G$ such that
the quotient map $\pi \colon \R^d \to \R^d/G$ is a smooth
submersion on $\R^d_\pr$.
Indeed, around any principal orbit there
is a geodesic slice
$\Sigma := \{x + v \colon v \in (T_x\mc O_x)^{\perp},
\, \, \|v\| < \eps \}$
for some $\epsilon > 0$, which
has the property that for all $y \in \Sigma$,
we have $G_x = G_y$~\cite[Section 2.1.6]{berndt2016submanifolds}.
The quotient $G/G_x$ is a smooth manifold~\cite[Theorem 21.17]{lee2013introduction},
and there is then a diffeomorphism $G\cdot \Sigma \to G/G_x \times \Sigma$.
Using this diffeomorphism, we can take coordinates $\phi \colon U \to \phi(U)$ around $x$ with the property
that $\phi(U)$ is an open cube in $\R^m \times \R^n$
and an orbit intersects $U$ either not at all or in a slice of the
last $n$ coordinates.
With these coordinates, the rest of the proof goes
in the standard way~\cite[Theorem 21.20]{lee2013introduction};
we sketch the details.

Let $V:= \pi(\Sigma)$; because $G_x = G_y$
for all $y \in \Sigma$, $\pi|_{\Sigma}$ is a bijection
onto its image. We then use the previous
coordinates to define the
smooth structure at $V$ by first passing through $\pi|_{\Sigma}^{-1}$.
In these coordinates, $\pi$ becomes a projection onto its
last $n$ coordinates, so is certainly smooth.
All that remains is to verify that the smooth structure
is compatible, but this follows from the
fact that the transition maps respect the orbits: points
are in the same orbit if and only if they have the same 
last $n$ coordinates, and so the transition maps
are smooth when restricted to a slice of the first $m$ coordinates.
With this structure, $\pi|_{\R^d_{\pr}}$ becomes a diffeomorphism
onto its image.
Using the maps $\psi:= \pi|_{\Sigma}^{-1}$ we define the
inner product of two tangent vectors $u,v \in T_p\R^d/G$
via $\langle u, v\rangle_p := \langle d\psi_p[u], d\psi_p[v]\rangle$;
this definition is independent of the choice of $\psi$
because $G$ acts by isometries.
This smooth structure and geometry is unique by the characteristic
property of smooth Riemannian submersions.

By abuse of notation, we write $\vol \mc O_p \colon \R^d/G \to \R$ for $\vol \mc O_{\pi^{-1}(p)}$.
With this geometry, we know that $\vol \mc O_p$ is a smooth function
on the image $\pi(\R^d_\pr)$, 
as it is simply the volume of the fibers of a smooth Riemannian
submersion whose fibers are compact, so is smooth.
By composition with $\pi$, we find that $\vol \mc O_x$ is smooth
at $x \in \R^d_{\pr}$.

Let us now turn to
the identity
$\nabla \log \vol \mc O_x = -H(x)$.
Let $\bar h(x) \colon \R^d \to \R$
be a smooth,
$G$-invariant function compactly supported within $\R^d_{\pr}$,
and write $\bar h = h \circ \pi$. By the previous part
of the proof, we know that the quotient map $\pi \colon \R^d \to \R^d/G$ 
is a smooth Riemannian submersion from $\R^d_{\pr}$ to $M := \pi(\R^d_\pr)$,
and since
$\bar h$ is supported in this set, the smooth co-area formula (see, e.g.~\cite[Section 2.6]{simon2014introduction}) yields
$$
\int \bar h(x) d x= \int h(p)\vol \mc O_p  d \vol_{M}(p).
$$
The idea is to use this fact combined with
integration by parts on $M$. Consider a smooth vector field $\bar V$ on $\R^d$,
compactly
supported within $\R^d_{\pr}$. Assume for now that
$\bar V \in (T_x \mc O_x)^{\perp}$ and $\bar V$ is $G$-equivariant,
so that $\bar V_{g\cdot x} = g\bar V_{x}$.
Then there exists a smooth vector field $V$ on $M := \pi(\R^d_\pr)$ such that $ V_{\pi(x)} = d\pi_x[\bar V_x]$ for all $x \in \R^d_{\pr}$.
We have
$$
\div(\bar V) = \div(V) \circ \pi - \langle \bar V, H_x \rangle.
$$ We include a proof of this
fact for the reader's convenience
as Lemma~\ref{lem:horizontal_divergence}
below.
Using this formula with the divergence theorem, we find that
\begin{align*}
    \int \langle \bar V, \nabla \log \vol \mc O_x \rangle d x
    &= \int \langle V, \nabla \vol \mc O_p \rangle d \vol_M(p)
    = -\int \div(V) \vol \mc O_p d \vol_M(p) \\
    &=-\int (\div(\bar V) + \langle \bar V, H_x \rangle) d x = 
    -\int \langle \bar V, H_x \rangle d x.
\end{align*} 
Given a general smooth vector field $W$,
compactly supported in $\R^d_{\pr}$,
we can transform it into a $G$-equivariant vector field
in $(T_x \mc O_x)^{\perp}$
by taking $\bar V(x) := \int g^\top W_{g\cdot x} d \mathrm{unif}_G(g)$,
where $\mathrm{unif}_G(g)$ denotes the normalized Haar measure on $G$.
Indeed, since $\nabla \log \vol \mc O_x \in (T_x \mc O_x)^{\perp}$
we can assume $W$ is too,
and notice that since $g$ acts by isometries we 
can use Proposition~\ref{prop:gradient_of_invariant_function}
to find that
$$
\int \langle W(x), \nabla \log \vol \mc O_x \rangle
d x = \iint \langle W(g\cdot x), \nabla \log \vol \mc O_{g\cdot x} \rangle 
d x d \mathrm{unif}_G(g)
= \int \langle \bar V(x), \nabla \log \vol \mc O_x \rangle d x.
$$ Using the analogous formula for $\int \langle  W, H_x\rangle d x$,
we obtain
$$
\int \langle W, \nabla \log \vol \mc O_x + H_x \rangle d x =0,
$$ for arbitrary smooth vector fields $W$ supported in $\R^d_{\pr}$.
Proposition~\ref{prop:smoothness_of_P}
implies that the mean curvature is a smooth function on $\R^d_{\reg}$,
and so the formula holds on $\R^d_{\reg}$
since $\R^d_{\pr}$ is dense in $\R^d$.
\end{proof}

\begin{lemma}[Divergence of horizontal
vector field]\label{lem:horizontal_divergence} Suppose
$\bar V$ is a smooth
vector field on an open set $S \subset \R^d_{\pr}$,
such that $\bar V_x \in (T_x\mc O_x)^{\perp}$ for all $x \in S$,
and assume there is 
a smooth vector field $V$ on
$\pi(S)$ such that $V_{\pi(x)}
= d\pi_x[\bar V]$.
Then
$$
\div(\bar V) = \div(V)\circ \pi 
- \langle \bar V, H\rangle,
$$ where $H$ is the mean curvature
vector field for the orbits $\mc O_x$.
\end{lemma}

\begin{proof}[Proof of Lemma~\ref{lem:horizontal_divergence}]
    Let $(\bar A_i)_{i = 1}^m$ be 
    a smooth local orthonormal frame for $(T_x \mc O_x)^\perp$ and $(B_j)_{j = 1}^n$ be a smooth local orthonormal frame for $T_x \mc O_x$.
    Using the fact, established in the proof of Proposition~\ref{prop:log_volume_mean_curvature},
    that the projection $\pi \colon \R^d_{\pr}
    \to \pi(\R^d_{\pr})$ is a smooth Riemannian
    submersion, we find that
    \begin{align*}
        \div(\bar{V}) & = \sum_{i =1 }^m \langle \nabla_{\bar{A}_i} \bar{V}, \bar{A}_i \rangle + \sum_{j = 1}^n \langle \nabla_{B_j} \bar{V}, B_j \rangle \, , \\
        & = \sum_{i = 1}^m \langle d \pi_x [\nabla_{\bar{A}_i} \bar{V}], d \pi_x[\bar{A}_i] \rangle_{\pi(x)} + \sum_{j=1}^n \langle \nabla_{B_j} \bar{V}, B_j\rangle \, .
    \end{align*}
    Now, let $\tilde{\nabla}$ be the Levi-Civita connection on $\pi(\R^d_\pr)$
    and write $A_i := d\pi[\bar V_i]$.
    By~\cite[Problem 5-6(b)]{lee2018introduction},
   we have $d \pi [\nabla_{\bar{A}_i} \bar{V}] = \tilde{\nabla}_{A_i} V$.
   Hence the first term is exactly $\sum_{i = 1}^m \langle \tilde{\nabla}_{A_i} V, A_i \rangle_{\pi(x)} = \div(V)\circ \pi$. For the second term, we 
   may calculate
    \begin{align*}
        \sum_{j = 1}^n \langle \nabla_{B_j} \bar{V}, B_j \rangle & = \sum_{j= 1}^n  B_j \cdot \langle \bar{V}, B_j \rangle - \sum_{j = 1}^n \langle \bar{V}, \nabla_{B_j} B_j \rangle  \, , \\
        & = - \langle \bar{V}, P(\sum_{j=1}^n \nabla_{B_j} B_j \rangle) \rangle \, , \\
        & = - \langle \bar{V}, H_x \rangle \, ,
    \end{align*}
    where in the second line we used the fact that 
    $\bar V \in (T_x \mc O_x)^{\perp}$ twice.
\end{proof}

\section{Dynamics remain $G$-invariant}
\label{appendix:dynamics_remain_g_invariant}
In this section we first collect
several useful calculations for $G$-invariant
functions, and then prove 
Lemma~\ref{lem:remain_G_invariant} on the
fact that the dynamics remain $G$-invariant.

\subsection{Useful calculations}
We first collect the following useful calculations,
which describe how gradients and the projection
matrices transform under the action of $G$.
\begin{prop}[Gradient of $G$-invariant functions]
\label{prop:gradient_of_invariant_function}
    Let $h$ be invariant under the action of $G$. Then for all $g \in G$ and all $x \in \mathbb{R}^d$, it holds that $g \cdot \nabla h(x) = \nabla h(g\cdot x)$.
\end{prop}
\begin{proof}[Proof of Proposition~\ref{prop:gradient_of_invariant_function}]
For all $g \in G$ and $x \in \R^d$, we have $h(g\cdot x) = h(x)$.
Thus
$$
\nabla h(x) = g^\top\nabla h(g\cdot x).
$$ But since $g \in O(d)$ by Assumption~\ref{assum:isometric},
$g^\top = g^{-1}$. Hence the result.
\end{proof}

\begin{prop}[Transformation of horizontal projection]
\label{prop:horizontal_proj}
    Let $P_x$ be the horizontal projection onto the 
    orbit $\mc O_x$ and $Q_x := I - P_x$.
    Then for all $x\in \mathbb{R}^d$ and $g \in G$:
    $$
 gP_x = P_{g\cdot x}g, \quad \quad 
    gQ_x = Q_{g\cdot x} g\,.
    $$
\end{prop}

\begin{proof}[Proof of Proposition~\ref{prop:horizontal_proj}]
Observe that if $\mathfrak{g}$ is the Lie algebra
of $G$, then $T_x \mc O_x = \{Ax \colon A \in \mathfrak{g}\}$.
But notice that, for each $g \in G$,
$\{gAg^{-1} \colon A \in \mathfrak{g} \} = \mathfrak{g}$, so that
for any $A \in \mathfrak{g}$ and $v \in \R^d$
$$
\langle g^{-1}P_{g \cdot x}gv,Ax \rangle
= \langle P_{g \cdot x}g v, gAx \rangle
= \langle g v, P_{g\cdot x} gAx \rangle=
\langle gv, P_{g\cdot x} gAg^{-1} (g\cdot x) \rangle = 0,
$$
where the first equality follows by the fact that $g \in O(d)$
by Assumption~\ref{assum:isometric}, the second equality
is by symmetry of $P$ since it is an orthogonal projection,
and the fourth equality is because $gAg^{-1} \in \mathfrak{g}$,
so that $gAx = gAg^{-1}(gx) \in T_{gx}\mc O_x$.
On the other hand, it is clear that $L = g^{-1}P_{g\cdot x} g$
is such that $L^2 = L$; and finally that 
$\mathrm{dim} L = \dim \mc O_x$.
This implies that $g^{-1}P_{g\cdot x} g = P_x$.
\end{proof}

\subsection{Proof of Lemma~\ref{lem:remain_G_invariant}}
\begin{proof}[Proof of Lemma~\ref{lem:remain_G_invariant}]
Under our assumptions,
Equation~\eqref{eqn:projected_SDE_full}
has a strong solution by standard SDE theory~\cite[Theorem 5.2.1]{oksendal2013stochastic}
and the solution is weakly unique~\cite[Lemma 5.3.1]{oksendal2013stochastic}.
It then suffices to show that $g \cdot X_t$ is a weak
solution of~\eqref{eqn:projected_SDE_full}
with the same initial condition for all $g \in G$,
since by weak uniqueness
it must be the same in law as $X_t$ itself,
establishing $G$-invariance.

Indeed, by $G$-invariance and Itô's lemma,
we may calculate that for any $g \in G$, we have
$$
d(g \cdot X_t)
= -g\cdot \nabla f(X_t)dt + g\cdot \sqrt{2}(\alpha(X_t)
P_{X_t} + \beta(X_t)Q_{X_t})dB_t.
$$ By Proposition~\ref{prop:gradient_of_invariant_function}
and Proposition~\ref{prop:horizontal_proj}
this is
$$
d(g \cdot X_t)
= -\nabla f(g \cdot X_t)dt
+ \sqrt{2}(\alpha(X_t)
P_{g \cdot X_t} +\beta(X_t)Q_{g \cdot X_t})
gdB_t.
$$ Finally using $G$-invariance of $\alpha$ and
$\beta$, we obtain
$$
d(g \cdot X_t)
= -\nabla f(g \cdot X_t)dt
+ \sqrt{2}(\alpha(g \cdot X_t)
P_{g \cdot X_t} + \beta(g\cdot X_t)Q_{g \cdot X_t})
gdB_t.
$$ But since $g \in O(d)$,
it follows that
$gdB_t$ is a Brownian motion, and hence
$g \cdot X_t$ is a weak solution of~\eqref{eqn:projected_SDE_full}. Because the initial
condition is $G$-invariant,
it follows that $X_t$ and $g \cdot X_t$
are weak solutions of the same SDE
and thus must have the same marginal laws. This proves
the first claim in the Lemma.

The proof of the second claim is analogous:
by Proposition~\ref{prop:log_volume_mean_curvature},
$\nabla \log \vol \mc O_x$, is smooth
away from the singular orbits. Given our
assumption on $\alpha - \beta$
is follows that 
$(\alpha^2 - \beta^2) \nabla \log \vol \mc O_x$ 
is smooth and compactly supported,
so is globally
Lipschitz. Hence the drift in
Equation~\eqref{eqn:full_isotropic_Langevin}
is globally Lipschitz,
and the diffusion matrix is globally Lipschitz
by assumption.
Applying again~\cite[Theorem 5.2.1]{oksendal2013stochastic}
and~\cite[Lemma 5.3.1]{oksendal2013stochastic}
we find there exists a strong solution to~\eqref{eqn:full_isotropic_Langevin}
and it is weakly unique. The remainder of the argument
is identical to the above and so is omitted.
\end{proof}

\section{Proof by SDE construction}
\label{appendix:SDE_proof}

\subsection{Proof of second fundamental form identity}
\label{subsec:second_fund_forms_eq}
\begin{proof}[Proof of Lemma~\ref{lem:second_fund_forms_eq}]
We first compute the second fundamental
form $h^G_g(A, A)$, for $A \in T_gG$. Let us begin
by considering the case where $g = I$. Given
any smooth vector field $W$ on $\R^d$ with
$W_I = A$ and such that $W_g \in T_gG$ for all $g\in G$, the second fundamental form is
$$
h^G_I(A, A) = P^G_I(\nabla_{W} W),
$$ where $P^G_I$ is the projection form $T_I\R^d$ to the normal
space $N_IG$ and $\nabla$ denotes the Euclidean connection.
We may therefore take $W$ to be $W_X:= XA$ for all $X \in \mathbb{R}^{d \times d}$,
which will be a 
smooth vector field on $G$ and such that $W_I = A$. Then
$$
(\nabla_{W}W)_I = \lim_{t\to 0} \frac{W_{I + tA} - W_{I}}{t}
= A^2.
$$ But then notice that, since $G \subset O(d)$, $A$ must be anti-symmetric,
so that in fact the above is a symmetric matrix. Since
$T_gG$ is a subset of anti-symmetric matrices, it follows that
$\nabla_W W \in N_IG$ already. Hence
$$
h^G_I(A, A) = A^2.
$$ We may then deduce the general formula by the fact that
$g \colon \R^{d \times d} \to \R^{d \times d}$ which takes
$X \mapsto gX$, is an isometry so that for $B \in T_gG$, we have
$$
h^G_g(B, B) = gh^G_I(g^{-1}B, g^{-1}B) = Bg^{-1}B.
$$

Next, we compute the second fundamental form of $\mc O_x$.
Here, we know that each $v \in T_x \mc O_x$ is of the form
$Ax$ for $A \in T_eG$. We can then extend $v$ to the vector
field $V(x') := Ax'$, and again obtain
$$
(\nabla_V V)_x = \lim_{t\to 0} \frac{V(x + tAx) - V(x)}{t}
= A^2x.
$$ Hence if $v = Ax$ then
\begin{align*}
h^{\mc O_x}_x(v, v)= P_x(\nabla_V V)_x = P_x(A^2x) = P_x(h^G_I(A, A)x)&=
P_x(g^{-1}h^G_g(gA, gA)x),
\end{align*}
yielding the result.
Applying Proposition~\ref{prop:horizontal_proj}, this implies
$$
gh_x^{\mc O_x}(v, v) = P_{g\cdot x}
(h_g^G(gA, gA)x).
$$ Since $g \colon \R^d \to \R^d$ is an isometry the left-hand-side
becomes
$$
h_{g\cdot x}^{\mc O_x}(g\cdot v, g \cdot v)
= P_{g\cdot x}(h_g^G(gA, gA)x).
$$ Because $v$ was an arbitrary element of $T_x \mc O_x$,
the result follows.
\end{proof}
\subsection{Proof of Theorem~\ref{thm:BM_on_orbit}}
\begin{proof}[Proof of Theorem~\ref{thm:BM_on_orbit}] Let
$L_{g_t, x}$ be as defined in~\eqref{eqn:Lgx_defn}. Put
\begin{align*}
J_0(g) &:= (L_{g_t, x}^\top  L_{g_t, x})^{\dagger/2} \\
V_0(g) &:= \tr(h_g^G(L^{\dagger}_{g, x},  L^{\dagger}_{g, x})), \\
V_1(g) &:= - L_{g, x}^{\dagger}Q_{g \cdot x}(V_0(g)x),
\end{align*} and
consider the SDE
\begin{equation}\label{eqn:orbit_BM_appendix}
dg_t := (V_0(g_t) + V_1(g_t))dt + \sqrt{2}J_0(g)dB_t',
\quad \quad g_0 = I.
\end{equation}
We first address the question of existence and uniqueness
of a strong solution to this SDE.
To begin, we claim that $J_0, V_0, V_1$ are all
smooth on $G$.
Indeed, $L_{g, x}$ is a smooth function
of $g$, and for all $g \in G$ it has rank $\dim \mc O_x$,
so that $L_{g,x}^{\dagger}$ and $J_0(g)$ are smooth~\cite{constales1998closed}.
Hence $V_0$ is also smooth
on $G$.
Similarly, $Q_{g\cdot x}$ is a smooth
function of $g$, and thus $V_1$ is smooth
as well.
Hence $J_0, V_0, V_1$ are smooth
on $G$,
and thus globally Lipschitz on $G$ by compactness.
With a minor abuse of notation, we may then assume that $J_0, V_0, V_1$
have been extended to globally Lipschitz functions on $\R^{d \times d}$.
Equation~\eqref{eqn:orbit_BM_appendix}
then indeed
has a unique strong solution~\cite[Theorem 5.2.1]{oksendal2013stochastic}.

We next verify that $g_t$ remains in $G$.
Let
$Z \colon \R^{d \times d} \to \R$ be a smooth function such that $Z(g) = 0$ if and only
if $g \in G$ (e.g. a smoothing of 
$d^2(x,G)$). Then observe that $\nabla Z(g) \in (T_gG)^{\perp}$ for all $g \in G$, and thus for all $v \in T_gG$,
\begin{equation}\label{eqn:hessian_2nd_fund_form}
\nabla^2 Z_g[v, v] + \langle h^G_g(v, v), \nabla Z(g) \rangle = 0.
\end{equation}
Using Itô's Lemma, we compute
\begin{align*}
dZ(g_t) &= (\langle \nabla Z(g_t), \tr(h_{g_t}^G(L^{\dagger}_{g_t, x}\cdot, L^{\dagger}_{g_t, x}\cdot)) \rangle + \tr(\nabla^2 Z(g_t) J_0(g_t)^2))dt,
\end{align*} where the $V_1$ and diffusion terms vanish because they
are in $T_{g_t}G$ and thus in the kernel of $dZ_{g_t}$. 
Using the SVD decomposition, we may write
$$
L_{g_t, x} = \sum_{i=1}^m \sigma_i u_i \langle A_i, \cdot \rangle,
$$
where $(u_i)_{i =1}^m$ are orthonormal vectors spanning $T_{g_t \cdot x} \mc O_x$
and $(A_i)_{i = 1}^m$ are orthonormal vectors in $T_{g_t}G$. Writing the above
using this basis, we find
$$
\langle \nabla Z(g_t), \tr(h_{g_t}^G(L^{\dagger}_{g_t, x}\cdot, L^{\dagger}_{g_t, x}\cdot)) \rangle
= \sum_{i = 1}^m \frac{1}{\sigma_i^2}\langle \nabla Z(g_t), h^G_{g_t}(A_i, A_i) \rangle,
$$ and
$$
\tr(\nabla^2 Z(g_t) J_0(g_t)^2) = \sum_{i = 1}^m \frac{1}{\sigma_i^2}\nabla^2 Z_{g_t}[A_i,A_i].
$$ Applying~\eqref{eqn:hessian_2nd_fund_form}, we thus find that $dZ(g_t) = 0$, and so combined with the initialization
at $g_0 = I$ we conclude that $g_t$ remains in $G$ at all times.

We now verify that $(g_t \cdot x)_t$
is a Brownian motion on 
$\mc O_x$.
We may
calculate using Itô's formula
and the definitions that
\begin{align*}
d(g_t \cdot x)
&=L_{g_t, x}(V_0(g_t) + V_1(g_t))
dt + \sqrt{2}L_{g_t, x}J_0(g_t)
dB_t' \\
&=P_{g_t, x}(V_0(g_t)x)dt
+\sqrt{2}L_{g_t, x}J_0(g_t)dB_t'.
\end{align*}
By Lemma~\ref{lem:second_fund_forms_eq},
this is exactly
$$
d(g_t \cdot x)
= H(g_t \cdot x)dt + \sqrt{2}L_{g_t, x}J_0(g_t)dB_t'.
$$
Notice that
$$
h_{g_t, x}J_0(g_t)J_0(g_t)^\top L_{g_t, x}^\top
= Q_{g_t\cdot x}.
$$ Hence by Theorem~\ref{thm:BM_on_orbit}
we have that $(g_t \cdot x)$ has the same
drift and diffusion as a Brownian
motion on $\mc O_x$.
Under our assumptions, weak uniqueness
holds for~\eqref{eqn:BM_on_orbit}~\cite[Theorem 9.4.3]{bogachev2022fokker},
and so
the marginal law of $(g_t \cdot x)_t$ must therefore
be identical to that of Brownian motion on $\mc O_x$.
\end{proof} 

\subsection{Full details of proof by constructing
a stochastic process on $G$}

\begin{proof}[Proof of Theorem~\ref{thm:main}]
{\it Proof approach.}
Let
$S := \supp(\alpha - \beta)$.
Since $S$ is compact and contained within
the set of regular orbits $\R^d_{\reg}$,
which is open by Proposition~\ref{prop:smoothness_of_P},
there exists a smooth bump function $\psi$
which is $1$ on $S$ but supported
on a compact set within $\R^d_{\reg}$~\cite[Proposition 2.25]{lee2013introduction}.
Let $\phi(x) := \E_{g\sim \unif_G}[\psi(g\cdot x)]$,
so that it is a $G$-invariant bump function for $S$.

Using this bump function $\phi$,
the proof of Theorem~\ref{thm:main}
by stochastic calculus consists in showing
that the two stochastic processes~\eqref{eqn:projected_SDE_full}
and~\eqref{eqn:full_isotropic_Langevin} are equivalent
in marginal law
to a {\it third} process:
\begin{equation}\label{eqn:auxiliary_SDE}
dZ_t = -(\nabla f + \phi^2(\alpha^2 + \beta^2)
\nabla \log \vol \mc O)dt 
+ \sqrt{2}(\alpha P
+ \sqrt{\beta^2 + \phi^2(\alpha^2 + \beta^2)}Q)dB_t,
\end{equation} where we suppress
$Z_t$ in the right-hand side.
We use this auxiliary
process because the ansatz of multiplying
$g_t \cdot X_t$ can only increase the amount of noise
in the $Q$ directions, at least when the Brownian motion of $g_t$
is independent of that of $X_t$.

{\it $X_t,Y_t,$ and $Z_t$ have unique strong solutions.} 
We begin by checking that both~\eqref{eqn:projected_SDE_full}
and~\eqref{eqn:full_isotropic_Langevin}
have a unique strong solution;
to this end, let
$S := \supp(\alpha - \beta)$.
First consider $X_t$, namely~\eqref{eqn:projected_SDE_full}.
The drift $\nabla f$ is globally Lipschitz by assumption,
and on $S^c$, $\alpha = \beta$, so
the diffusion matrix is $\alpha I$ which is globally Lipschitz
by assumption. Since $S \subset \R^d_{\reg}$,
Proposition~\ref{prop:smoothness_of_P}
implies that
the matrix $\alpha P + \beta Q$ is smooth;
since $S$ is compact it must be Lipschitz on $S$,
and therefore it is globally Lipschitz.
The equation~\eqref{eqn:projected_SDE_full}
for $X_t$
therefore has a unique strong solution by~\cite[Theorem 5.2.1]{oksendal2013stochastic}.

For $Y_t$, namely~\eqref{eqn:full_isotropic_Langevin},
we again check global Lipschitz-ness of the
drift and diffusion matrix. Here, the drift is globally
Lipschitz because $\nabla f$ is by assumption, and because
the additional term $(\alpha^2 - \beta^2)\nabla \log\vol \mc O$
is smooth on $\R^d_{\reg}$ by 
Proposition~\ref{prop:log_volume_mean_curvature}
and it is supported in $S\subset \R^d_{\reg}$.
The diffusion matrix is $\alpha I$ which
is globally Lipschitz because
$\alpha$ is.
The equation~\eqref{eqn:full_isotropic_Langevin}
therefore has a unique strong solution by~\cite[Theorem 5.2.1]{oksendal2013stochastic}.

For $Z_t$ in~\eqref{eqn:auxiliary_SDE}, we 
verify that the drift and diffusion are globally
Lipschitz. 
Here, the drift is globally
Lipschitz because $\nabla f$ is by assumption, and because
the additional term $\phi^2(\alpha^2 + \beta^2)\nabla \log\vol \mc O$
is smooth on and supported in $\R^d_{\reg}$ by 
Proposition~\ref{prop:log_volume_mean_curvature}.
For the diffusion matrix, first observe that
$$
\beta^2 + \phi^2(\alpha^2 + \beta^2)=
\alpha^2 + 2\phi^2\beta^2,
$$ and since we assume $\alpha > 0$ everywhere,
the square root is smooth.
Hence by Proposition~\ref{prop:smoothness_of_P}
the diffusion matrix is smooth on $\R^d_{\reg}$
and thus smooth everywhere. It is globally
Lipschitz because it is $\alpha I$ outside the support of $\phi$
and otherwise smooth on a compact set.
The equation~\eqref{eqn:auxiliary_SDE}
for $Z_t$
therefore has a unique strong solution by~\cite[Theorem 5.2.1]{oksendal2013stochastic}.
In particular, this solution must be $G$-invariant
by the same argument as in the proof of Lemma~\ref{lem:remain_G_invariant}.

{\it Equivalence of $X_t$ and $Z_t$.} Let us begin by showing that~\eqref{eqn:projected_SDE_full}
is equivalent in marginal law to~\eqref{eqn:auxiliary_SDE}.
Recall the definition of $L_{g,x}$ from~\eqref{eqn:Lgx_defn},
and put
\begin{align*}
V_0(g, x) &:=\phi^2(g\cdot x)(\alpha^2(g\cdot x) + \beta^2(g\cdot x)) \tr(h_{g}^G(L_{g, x}^{\dagger}
    \cdot, L_{g, x}^{\dagger} \cdot)),\\
    V_1(g, x) &:= -L_{g, x}^{\dagger}(V_0(g, x)x), \\
    J_0(g, x) &:= \phi(g\cdot x) \sqrt{\alpha(g\cdot x)^2 + \beta(g \cdot x)^2} \cdot
    (L_{g, x}^\top L_{g, x})^{\dagger/2}.
\end{align*}
We then consider the system of SDEs
\begin{align}
dX_t &= -\nabla f(X_t)dt + \sqrt{2}(\alpha(X_t) P_{X_t}
+ \beta(X_t) Q_{X_t}) dB_t, \quad \quad X_0 \sim \mu \notag \\
    dg_t &= (V_0(g_t, X_t) + V_1(g_t, X_t))dt
    + \sqrt{2}J_0(g_t, X_t)dB_t', \quad \quad g_0 = I.
    \label{eqn:full_gt_dynamics}
    \end{align}

We first verify that equation~\eqref{eqn:full_gt_dynamics}
has a unique strong solution.
Observe that $L_{g, x}$ is a smooth function
of $g$ and $x$, and for all $g \in G$
and $x \in \R^d_{\reg}$,
it has the same rank.
It follows that $L_{g,x}^{\dagger}$
and $(L_{g,x}^\top L_{g,x})^{\dagger/2}$
are smooth for $x \in \R^d_{\reg}$
and $g \in G$~\cite{constales1998closed}.
Hence $V_0$ and $V_1$ are smooth
on $G \times \R^d_{\reg}$.
The matrix $J_0$ is also smooth on $G \times \R^d_{\reg}$
because we assume $\alpha > 0$ everywhere.
But because $\phi$ is supported in $\R^d_{\reg}$,
we find that $V_0, V_1, J_0$ are smooth
on all of $G \times \R^d$.
Because $G$ and $\supp(\phi)$ are compact,
they must be globally Lipschitz in $G \times \R^d$,
and thus can be extended
to globally Lipschitz functions on all of $\R^d \times \R^{d \times d}.$
By the same argument as before, the drift and diffusion
for $dX_t$ are globally Lipschitz as well.
Hence~\eqref{eqn:full_gt_dynamics} has a unique strong solution.

The same argument as in the proof of Theorem~\ref{thm:BM_on_orbit}
shows that $g_t$ then remains in $G$ at all times.
For property P2, we calculate using Itô's Lemma that if $g_t$ follows~\eqref{eqn:full_gt_dynamics}
and $X_t$ follows~\eqref{eqn:projected_SDE_full},
then
\begin{align*}
d(g_t \cdot X_t) &= 
(-g_t\nabla f(X_t) + L_{g_t, X_t}(V_0(g_t, X_t)
+ V_1(g_t, X_t)))dt \\
&+ g_t\sqrt{2}(\alpha(g_t \cdot X_t)
P_{X_t} + \beta(g_t \cdot X_t)Q_{X_t})dB_t
+ L_{g_t, X_t}J_0(g_t, X_t)dB_t'.
\end{align*}
Observe that
$$
L_{g, x}(V_0(g, x)
 + V_1(g, x)) = P_{g\cdot x}(V_0(g, x)x).
$$ Taking an orthonormal basis $v_1,\ldots, v_m$
for $T_{g\cdot x}\mc O_x$ as well as
orthonormal $A_1, \ldots, A_m \in T_gG$ such that $L_{g, x}A_i=\sigma_i v_i$, we find
by Lemma~\ref{lem:second_fund_forms_eq} that
\begin{align*}
P_{g\cdot x}(V_0(g, x)x) &=
\phi(x)^2(\alpha(x)^2 + \beta(x)^2)
\sum_{i = 1}^m \frac{1}{\sigma_i^2}P_{g\cdot x}(h_g^G(A_i, A_i)x)\\
&=\phi(x)^2(\alpha(x)^2 + \beta(x)^2)
\sum_{i = 1}^m h_{g\cdot x}^{\mc O_x}(v_i, v_i) \\
&= \phi(x)^2(\alpha(x)^2 + \beta(x)^2)H(g\cdot x)
= \phi(g\cdot x)^2(\alpha(g\cdot x)^2 + \beta(g \cdot x)^2)
H(g\cdot x).
\end{align*}
By Proposition~\ref{prop:log_volume_mean_curvature}, $H(g \cdot x)
= -\nabla \log \vol \mc O_{g \cdot x}$ for $x \in \R^d_{\reg}$, so that
\begin{align*}
d(g_t \cdot X_t)
&= - (\nabla f(g_t\cdot X_t)
+ \phi(g_t \cdot X_t)^2 (\alpha(g_t \cdot X_t)^2 + \beta(g_t \cdot X_t)^2)\nabla \log \vol \mc O_{g_t \cdot X_t})dt \\
&+g_t\sqrt{2}(\alpha(X_t)
P_{X_t}
+ \beta(X_t)Q_{X_t})dB_t + L_{g_t, X_t}J_0(g_t, X_t)dB_t'.
\end{align*}
Now, observe that by our choice of $J_0$,
$$
L_{g, x}J_0(g, x) J_0(g, x)^\top L_{g, x}^\top
= \phi(g\cdot x)^2(\alpha(g\cdot x)^2 + \beta(g\cdot x)^2)
Q_{g \cdot x}
$$
The covariance is therefore
\begin{align*}
&g_t(\alpha(X_t)^2P_{X_t} + \beta(X_t)^2Q_{X_t})g_t^{-1}
+
\phi(X_t)^2(\alpha(X_t)^2 + \beta(X_t)^2)
Q_{g_t \cdot X_t}\\
&= \alpha(X_t)^2P_{g_t\cdot X_t}
+ \beta(X_t)^2 Q_{g_t\cdot X_t} + \phi(X_t)^2(\alpha(X_t)^2 + \beta(X_t)^2)
Q_{g_t \cdot X_t}.
\end{align*}
where the equality follows by Proposition~\ref{prop:horizontal_proj}.
Observe that for any $g \in G$,
with $\tilde X_0 := g \cdot X_0$
By~\cite[Theorem 9.4.3]{bogachev2022fokker},
the marginal law of $(g_t \cdot X_t)_t$ must be 
identical to that of $Z_t$ solving~\eqref{eqn:auxiliary_SDE},
so that, in particular, it is $G$-invariant.

To conclude, we argue that the law of $X_t$ must be the
same as that of $g_t \cdot X_t$,
and thus $Z_t$, since both are $G$-invariant.
Given a test function $h\colon \R^d \to \R$, let $\bar h(x) := \E_{g \sim \mathrm{unif}_G}[h(g\cdot x)]$
where $\mathrm{unif}_G$ denotes the normalized Haar measure
on $G$. Observe that $\bar h(g\cdot x) = \bar h(x)$ for any $g\in G$.
Since
$(g_t \cdot X_t)$ and $X_t$ are both $G$-invariant, we have
$$
\E[h(Z_t)] = 
\E[h(g_t \cdot X_t)] = \E[\bar h(g_t \cdot X_t)]
= \E[ \bar h(X_t)] = \E[h(X_t)].
$$ Since $h$ was an arbitrary test function, it follows that
$X_t$ and $Z_t$ have the same marginal distributions at all times.

{\it Equivalence of $Y_t$ and $Z_t$.}
This proof is completely
analogous to that of $X_t$ and $Z_t$, with the only
difference being the scaling of $J_0, V_0$
for $g_t$. We let
\begin{align*}
V_0(g, x) &:=2\phi(g\cdot x)^2\beta(g\cdot x)^2 \tr(h_{l}^G(L_{g, x}^{\dagger}
    \cdot, L_{g, x}^{\dagger} \cdot)),\\
    V_1(g, x) &:= -L_{g, x}^{\dagger}(V_0(g, x)x), \\
    J_0(g, x) &:= \sqrt{2} \phi(g\cdot x)\beta(g \cdot x) 
    (L_{g, x}^\top L_{g, x})^{\dagger/2}.
\end{align*}
And then consider the system of SDEs
\begin{align}
dY_t &= -(\nabla f(Y_t) + (\alpha(Y_t)^2
- \beta(Y_t)^2) \nabla \log \vol \mc O_{Y_t})
dt + \sqrt{2}\alpha(X_t) PdB_t, \quad \quad Y_0 \sim \mu \notag \\
    dg_t &= (V_0(g_t, Y_t) + V_1(g_t, Y_t))dt
    + \sqrt{2}J_0(g_t, Y_t)dB_t', \quad \quad g_0 = I.
    \label{eqn:full_gt_dynamics2}
    \end{align}
    As before, we calculate that the drift of $g_t \cdot Y_t$ 
    is
    $$
    -\nabla f(g_t \cdot Y_t)
    - (\alpha(g_t \cdot Y_t)^2 - \beta(g_t \cdot Y_t)^2 + 2\phi(g_t \cdot Y_t)^2 \beta(g_t \cdot Y_t)^2)\nabla \log \vol \mc O_{g_t \cdot Y_t}.
    $$ But since $\alpha^2 - \beta^2 = \phi^2(\alpha^2 - \beta^2)$
    we have
    $$
    \alpha^2 - \beta^2 + 2\phi^2 \beta^2
    = \phi^2 (\alpha^2 - \beta^2) + 2\phi^2 \beta^2
    = \phi^2(\alpha^2 + \beta^2),
    $$ hence the drift is the same as that of $Z_t$.
    As before, we calculate
    that the covariance matrix of $g_t \cdot Y_t$ is
    $$
    \alpha(g_t \cdot Y_t)^2
    P_{g_t \cdot X_t}
    + (\alpha(g_t \cdot Y_t)^2 + 2 \phi^2(g_t \cdot Y_t)
    \beta(g_t \cdot Y_t)^2)Q_{g_t \cdot Y_t}.
    $$ This is the same as the covariance matrix of $Z_t$ since
    $$
    \alpha^2 + 2\phi^2 \beta^2 = 
    \beta^2 + \alpha^2 - \beta^2 + 2\phi^2\beta^2=
    \beta^2 + \phi^2(\alpha^2 - \beta^2) + 2\phi^2\beta^2
    = \beta^2 + \phi^2(\alpha + \beta^2).
    $$ The rest of the proof is identical to the previous one:
    $Z_t$ and $g_t \cdot Y_t$ must have the same
    marginal laws, and by $G$-invariance we can conclude
    that $Y_t$ has the same marginal law as $g_t \cdot Y_t$
    and thus $Z_t$.

    {\it Conclusion.} We have shown that $X_t$ and $Y_t$
    have the same marginal law as the auxiliary process $Z_t$,
    and thus $X_t$ and $Y_t$ have the same marginal law,
    yielding Theorem~\ref{thm:main}.
\end{proof}

\section{Proof by PDE analysis}
\label{appendix:proof_by_PDE}

\begin{proof}[Proof of Theorem~\ref{thm:main} by PDE analysis]
We verified in the proof by SDE coupling
that both~\eqref{eqn:projected_SDE_full}
and~\eqref{eqn:full_isotropic_Langevin}
have unique strong solutions,
so we do not repeat that argument here.

Let the law of $X_t$ be
$\rho_t$.
By Itô's formula,
$\rho_t$ is a solution of the {\it weak Fokker--Planck equation}:
for all positive times $T > 0$ and all smooth compactly supported $\phi \colon \R^d \times (0, T) \to \R$,
\begin{equation}\label{eqn:weak_fokker_planck}
\int_{\R^d\times (0, T)} (\partial_t \phi + 
\tr(R\nabla^2\phi) - \langle \nabla f, \nabla \phi \rangle) \rho_t = 0,
\end{equation} where for ease of notation, we let $R := \alpha^2 P + \beta^2 Q$.
The idea is then to show that this $\rho_t$ also
satisfies the weak Fokker--Planck equation associated
to the SDE~\eqref{eqn:full_isotropic_Langevin}.

Suppressing the time $t$, observe that
$$
 \nabla \phi(g\cdot x) = g \nabla (\phi \circ g)(x),\quad \quad 
 \nabla^2 \phi(g \cdot x) = g^\top \nabla^2 (\phi \circ g)(x) g.
$$
So by Lemma~\ref{lem:remain_G_invariant} we have
\begin{align*}
\int_{\R^d} \tr(R\nabla^2 \phi) &- \langle \nabla f, \nabla \phi\rangle )\rho_t \\
&=  
\int_{\R^d} \int (\tr(R_{g\cdot x} \nabla^2 \phi(g \cdot x))
- \langle \nabla f(g\cdot x), \nabla \phi(g \cdot x) \rangle)
\rho_t(x)
d \unif_G(g)d x.
\end{align*}
Applying Prop.~\ref{prop:gradient_of_invariant_function} and
Prop.~\ref{prop:horizontal_proj}
this is exactly
$$
\int_{\R^d} \tr(R\nabla^2 \phi) - \langle \nabla f, \nabla \phi\rangle )\rho_t
= \int_{\R^d} (\tr(R \nabla^2 \bar \phi)
-\langle \nabla f, \nabla \bar \phi \rangle) \rho_t,
$$ where $\bar \phi(x) := \int \phi(g\cdot x) d \unif_G(g)$.
We now come to the main
geometric observation driving the current
proof, we remark
that a similar observation
appears in~\cite[Theorem 3.7]{helgason1984groups}.
\begin{lemma}[Hessian of $G$-invariant functions]
\label{lem:hessian_G_invariant}
If $\phi \colon \R^d \to \R$ is $G$-invariant
and $v \in T_x \mc O_x$
then
$$
\nabla^2 \phi[v, v] = - \langle \nabla \phi, h_x(v, v)\rangle.
$$
\end{lemma}
Using this result, we obtain
$$
\int_{\R^d\times (0, T)} (\partial_t \bar \phi + 
\alpha^2 \Delta \bar \phi - \langle \nabla \bar \phi,
(\beta^2 - \alpha^2) H +  \nabla f \rangle) \rho_t = 0.
$$
Since $H(x) = - \nabla \log \vol \mc O_x$ by Prop.~\ref{prop:log_volume_mean_curvature}
and the volume $\vol \mc O_x$ is $G$-invariant,
we can again apply Prop.~\ref{prop:gradient_of_invariant_function}
to go from $\bar \phi$ to $\phi$ and obtain
$$
\int_{\R^d\times (0, T)} (\partial_t  \phi + 
\alpha^2 \Delta \phi - \langle \nabla  \phi,
(\alpha^2 - \beta^2) \nabla \log \vol \mc O +  \nabla f \rangle) \rho_t = 0.
$$
This is the weak Fokker--Planck equation associated
to equation~\eqref{eqn:full_isotropic_Langevin}.
But by Itô's Lemma, this equation is also satisfied by
the law of $Y_t$ of~\eqref{eqn:full_isotropic_Langevin}.
The result~\cite[Theorem 9.4.3]{bogachev2022fokker}
ensures that there can only be one such solution, yielding the result.
\end{proof}

\begin{proof}[Proof of Lemma~\ref{lem:hessian_G_invariant}]
By taking appropriate local coordinates,
one can construct a smooth vector field $V$ on $\R^d$ 
such that $V_x = v$ and $V_y \in T_y \mc O_y$
for all $y$ in a neighborhood of $x$.
Then
$$
\nabla^2 \phi[v,v]=
\langle \nabla_V \nabla \phi, V \rangle 
= \nabla_V \langle \nabla \phi, V \rangle - \langle \nabla \phi, 
\nabla_V V \rangle = -\langle \nabla \phi, \nabla_V V\rangle
= -\langle \nabla \phi, h_x(v, v)\rangle.
$$
\end{proof}

\subsection{Proof of Corollary~\ref{cor:identity_diffusion}}
\begin{proof}[Proof of Corollary~\ref{cor:identity_diffusion}]
Under the assumptions of Corollary~\ref{cor:identity_diffusion},
the SDE~\eqref{eqn:projected_SDE_identity} 
has smooth drift and diffusion matrices,
and thus locally Lipschitz drift and diffusion matrices.
Because we assume that $\phi$ is bounded,
the diffusion matrix is also bounded in Frobenius
norm.
We may therefore apply~\cite[Proposition 4.2.4]{kunze2012sde}
to conclude that Equation~\eqref{eqn:projected_SDE_identity}
has a unique
strong solution.
This implies weak uniqueness,
and thus the proof
of Lemma~\ref{lem:remain_G_invariant} goes through as before to show that $X_t$ remains
$G$-invariant at all times.
On the other hand, our assumptions guarantee
that the SDE
converge to a unique stationary distribution, 
which must be $G$-invariant.
It is straightforward to check using
the calculations in the PDE proof that this
stationary distribution is the claimed one,
proving the result.
\end{proof}

\section{Proofs for examples}
\label{sec:examples}

\subsection{Auxiliary lemmas}

We collect several elementary linear algebra results that will prove useful for volume computations.

\begin{lemma}
\label{lemma:prod_diagonal}
    Denote $E_{ij}$ the canonical basis of $\mathbb{R}^{d\times d}$ and $D$ a diagonal matrix with entries $\lambda_k$. It holds that
    \begin{align*}
        & E_{ij}D = \lambda_i E_{ij}\, , \\
        & D E_{ij} = \lambda_j E_{ij} \, .
    \end{align*}
    
\end{lemma}

\begin{proof}
    We have
    \begin{equation*}
     [E_{ij}D ]_{kl} = \sum_{p=1}^d (E_{ij})_{kp} D_{pl} 
      = \sum_{p=1}^d (E_{ij})_{kp} \lambda_k \delta_{p=l} = (E_{ij})_{kl} \lambda_k  = \lambda_k \delta_{(i,j) = (k,l)} \, .
    \end{equation*}
  Hence $E_{ij}D = \lambda_i E_{ij}$  and by taking the transpose, $D E_{ij} = \lambda_j E_{ij}$.
\end{proof}

\begin{lemma}
    \label{lemma:tangent_space_O(d)}
    Let $O \in O(d)$. We have $T_O O(d) = \{ OA ~| ~ A^\top = -A\} = \{ AO ~| ~ A^\top = -A\}$.
\end{lemma}
\begin{proof}
    Let $O(t)$ be a curve on $O(d)$ such that $O(0)=O$. We have that $O(t)^\top O(t) = I_d$, hence differentiating w.r.t. $t$ yields 
    $$
    O^\top \dot{O}(0) + \dot{O}(0)^\top O = 0 \, .
    $$
    A necessary and sufficient condition on $\dot{O}(0)  $ is thus $\dot{O}(0) = OA$ with $A$ a skew matrix. 
\end{proof}

\begin{lemma}
\label{lem:parallelizable}
Suppose $(M, g)$ is a smooth
Riemannian
manifold with or without boundary
and $\tilde g$ is another
metric tensor such that $(M, \tilde g)$
is also a smooth Riemannian manifold.
Suppose that there exists a global
frame $V_1, \ldots, V_n$
which is orthonormal 
for $g$ and such that 
$\tilde g(V_i, V_j) = c_i\delta_{ij}$ for all $i, j \in [n]$
for $c_i > 0$ constant
on $M$.
Then 
$$
\vol(M, \tilde g) =\big( \prod_{i = 1}^n
\sqrt{c_i}\big) \vol(M, g).
$$
\end{lemma}

\begin{proof} Let us begin
by assuming that the manifold $(M, g)$
is oriented.
Let $\omega^1, \ldots, \omega^n$ denote
the covectors corresponding to
$V_1, \ldots, V_n$. If necessary,
swap labels such that $\omega^1 \land \cdots \land \omega^n$ is positively oriented.
Define
$W_i := \frac{1}{\sqrt{c_i}} V_i$. Then note
that
$$
\tilde V:= \big(\prod_{i= 1}^n \sqrt{c_i}\big) \omega^1 \land \cdots 
\omega^n
$$ evaluates to $1$ at $(W_1, \ldots, W_n)$.
Because $(W_1, \ldots, W_n)$ is an orthonormal
frame, $V$ must be the Riemannian volume form on
$(M, \tilde g)$~\cite[Proposition 15.29]{lee2013introduction}.
But the Riemannian volume form
on $(M, g)$ is $V:=\omega^1 \land \cdots \land \omega^n$, so that 
$$
\vol(M, \tilde g) = \int \tilde V = 
\prod_{i= 1}^n \sqrt{c_i} \int V = 
\big(\prod_{i =1 }^n \sqrt{c_i} \big) \vol(M, g).
$$

To extend to non-orientable manifolds, we may take
a partitition of unity and apply
the above formula on each partition.
\end{proof}

\subsection{Conjugation over symmetric matrices}

For the action $(O \in O(d), M \in \text{Sym}(d)) \mapsto O\cdot M = O^\top M O $, the orbits are given by the set of eigenvalues: two matrices are in the same orbit if and only if they have the same eigenvalues. In particular, the principal orbits are the symmetric matrices with $d$ distinct eigenvalues. Let us compute the volume of these principal orbits. We
first show that an orbit $\mc O_D$ with $D$ a diagonal matrix with $d$ distinct eigenvalues, can be parametrized by a subset of $O(d)$ via $\Phi$ defined as 
\begin{equation*}
    \begin{cases}
     \Phi : &\mathcal{F} \to O_D \, , \\
       & O \mapsto O^\top D O \, ,
    \end{cases}
\end{equation*}
with $\mathcal{F} = \{ O \in O(d) ~ | ~\text{the first non-zero element of each column is positive}\}$. Let us prove that $\Phi$ is indeed invertible.

\emph{Injectivity.} Assume that there exists $(O_1, O_2) \in \mathcal{F}^2$ such that $O_1^\top D O_1 = O_2^\top D O_2$. In particular it holds that $O_2O_1^\top D = D O_2 O_1^\top$. Denoting $O = O_2 O_1^\top$ and $(e_i)$ the canonical basis, we decompose $O e_i$ as $\sum_{k=1}^d \alpha_k e_k$. Applying the previous equality yields $\sum_{k=1}^d \alpha_k e_k(\lambda_k - \lambda_i) = 0$. Hence, since all eigenvalues are distinct, we recover that $\alpha_k = 0$ for $k \neq i$ and $O e_i = \alpha_i e_i$. Since $\| O e_i \| =1$, we get that $\alpha_i = \pm 1$ so that $O$ is a diagonal matrix with only $\pm 1$ entries that we shall denote $\Delta$. In particular we get $O_1 = \Delta O_2$. Since both $O_1, O_2$ belong to $\mathcal{F}$, the sign constraint implies $\Delta = I_d$.

\emph{Surjectivity.} For $M \in \mc O_D$, there exits $O \in O(d)$ such that $M = O^\top D O$. Now, for all $\Delta$ a diagonal matrix with $\pm 1$ entries, it holds that $\Delta^\top D \Delta = D$ so that $M = (\Delta O)^\top D \Delta O$ for all $\Delta$. In particular, $\Delta$ can be chosen such that $\Delta O \in \mathcal{F}$. 

Hence we have $\vol{\mc O_D} = \vol{\mathcal{F}}$ with the metric on $\mathcal{F}$ given by the pullback metric $\tilde{g}_{O}$ that reads for all $(HO, VO) \in T_O\mathcal{F}$ with $(H,V)$ skew matrices,
\begin{align*}
    \tilde{g}_{O}(HO, VO) & = \tr(d \Phi_O(OH), (d \Phi_O(OV))^\top) \, , \\
    & = \tr((O^\top H^\top DO + O^\top DH O)(O^\top V^\top D O +O^\top DVO)) \, ,\\
    &= \tr(O^\top H^\top DV^\top D O + O^\top H^\top D^2 VO + O^\top DH V^\top D O + O^\top DH DV O) \, , \\
    &= \tr(HDVD - HD DV - DH VD + DH DV) \, , \\
    & = \tr([H, D][V, D]) \, ,
\end{align*}
where we denoted $[\cdot, \cdot]$ the Lie bracket $[H, D] = HD - DV$.
Let us compute this metric in the basis $\tilde{G}_{ij} = (E_{ij} - E_{ji})O$ of $T_O\mathcal{F}$ (see Lemma \ref{lemma:tangent_space_O(d)}). First we compute the Lie bracket
\begin{align*}
    [E_{ij} - E_{ji}, D] & = (E_{ij} - E_{ji})D - D(E_{ij} - E_{ji}) \, , \\
    & = (\lambda_i E_{ij} - \lambda_j E_{ji}) - (\lambda_j E_{ij} - \lambda_i E_{ji}) \, , \\
    & = (E_{ij} + E_{ji})(\lambda_i - \lambda_j) \, .
\end{align*}
It follows that 
\begin{align*}
    \tilde{g}_{O}(\tilde{G}_{ij}, \tilde{G}_{kl}) & = (\lambda_i - \lambda_j)(\lambda_k - \lambda_l)\tr((E_{ij} + E_{ji})(E_{kl} + E_{lk})) \, ,\\ 
    & = (\lambda_i - \lambda_j)(\lambda_k - \lambda_l)\tr(\delta_{j=k}E_{il} + \delta_{i=} E_{jk} + \delta_{i=k}E_{jl} + \delta_{j=l}E_{ik}) \, , \\
    & = 2(\lambda_i - \lambda_j)^2\delta_{(i,j)=(k,l)} \, .
\end{align*}
Observe also that the vector fields $\frac{1}{\sqrt{2}}\tilde G_{ij}$
are orthonormal with respect to
the standard metric on $\mathcal{F}$.
Hence we can apply Lemma~\ref{lem:parallelizable}
to conclude that
$$
\vol \mc O_D \propto \prod_{i<j} |\lambda_i - \lambda_j| \vol(\mathcal{F}) \, .
$$

\subsection{Bures-Wasserstein}

Consider the action over the space of squared matrices $X \in \mathbb{R}^{d\times d} \mapsto XO$. Let us compute the volume of the orbit at a given invertible matrix $X$:  $\mc O_X = \{XO ~ | ~ O \in O(d)\} $. In that case, 
\begin{equation*}
    \begin{cases}
     \Phi : & O(d) \to \mc O_X \, , \\
       & O \mapsto XO \, ,
    \end{cases}
\end{equation*}
is a diffeomorphism. In particular, $\vol(\mc O_X) = \vol(O(d))$ where $O(d)$ is equipped with the pullback metric. Let us compute the pullback metric for well-chosen coordinates. Denote $X^\top X = P^\top \Sigma P$ with $P \in O(d)$ and $\Sigma$ diagonal with entries $(\sigma_i^2)$ and consider the following basis of $T_O O(d)$: $\tilde{G}_{ij} = P^\top(E_{ij} - E_{ji})P O$ with $i<j$. In this basis,
applying Lemma~\ref{lemma:prod_diagonal}
implies that the pullback metric reads
\begin{align*}
    \tilde{g}_{ij, kl} & = \tr(X\tilde{G}_{ij} (X\tilde{G}_{kl})^\top) \, ,\\
    & = -\tr(P^\top \Sigma PP^\top(E_{ij} - E_{ji})P O O^\top P^\top(E_{kl} - E_{lk})PP^\top \Sigma P) \, , \\
    & = -\tr((E_{ij} - E_{ji}) \Sigma^2 (E_{kl} - E_{lk})) \, , \\
    & = -\tr((\sigma_i^2E_{ij} - \sigma_j^2 E_{ji})(E_{kl} - E_{lk})) \, , \\
    & = -\tr(\sigma_i^2E_{ij} E_{kl} - \sigma_i^2E_{ij}E_{lk} - \sigma_j^2 E_{ji}E_{kl} +  \sigma_j^2 E_{ji}E_{lk}) \, , \\
    & = -\tr(\sigma_i^2 \delta_{j=k} E_{il} - \sigma_i^2 \delta_{j=l}E_{ik} - \sigma_j^2 \delta_{i=k} E_{jl} + \sigma_j^2 \delta_{i=l}E_{jk}) \, ,\\
    & = - \sigma_i^2 \delta_{(j,i)=(k,l)} + \sigma_i^2 \delta_{(i,j)=(k,l)} + \sigma_j^2 \delta_{(i, j)=(k,l)} -  \sigma_j^2 \delta_{(i, j)=(l,k)} \, , \\
    & = (\sigma_i^2 + \sigma_j^2)\delta_{(i,j)=(k,l)} \, .
\end{align*}
Observe that $\tilde G_{ij}$ are
also a smooth orthonormal frame for $O(d)$
with respect to the standard metric.
Hence we may apply Lemma~\ref{lem:parallelizable}
to determine that 
$$
\vol(\mc O_X) 
    \propto \prod_{i<j} \sqrt{\sigma_i^2 + \sigma_j^2} \vol (O(d)) \, .
$$

\subsection{Brownian motion on the group does not necessarily yield Brownian on the orbit}
\label{subsec:BM_on_group_counterexample}

Consider the case $G = SO(2)$,
acting by left multiplication on $\R^2$. Brownian motion on $G$ is explicitly given by
$$
g_t = 
\begin{pmatrix}
\cos(\theta_t) & -\sin(\theta_t) \\
\sin(\theta_t) & \cos(\theta_t)
\end{pmatrix}
\, ,
$$
with $\theta_t$ a 1D Brownian motion~\cite[Example 3.3.1]{hsu2002stochastic}. Now for $x$ fixed, consider $y_t = g_t \cdot x$. The process $y_t$ verifies 
$$
d y_t = - y_t dt + \sqrt{2} \begin{pmatrix}
-\sin(\theta_t) & - \cos(\theta_t) \\
\cos(\theta_t) & -\sin(\theta_t)
\end{pmatrix} x d \theta_t \, .
$$
Hence, while the drift term $-y_t$ does correspond to mean curvature, the diffusion term, whose norm grows with $\|x\|$, cannot be expressed as $Q_{y_t} d B_t$ with $B_t$ a 2D Brownian motion since its norm is unitary and independent of $x$.

\end{document}